\documentclass{amsart}
\usepackage{amssymb}
\usepackage[mathscr]{eucal}
\usepackage{enumerate}
\usepackage[all]{xy}
\usepackage{enumitem}

\newcommand{\script}[1]{{\mathcal{#1}}}
\newcommand{\A}{\script{A}}
\newcommand{\B}{\script{B}}
\newcommand{\I}{\script{I}}
\newcommand{\Zz}{\script{Z}}
\newcommand{\Y}{\script{Y}}
\newcommand{\W}{\script{W}}
\newcommand{\frah}{\mathfrak{H}}
\newcommand{\frack}{\mathfrak{K}}
\newcommand{\Zme}{\mathcal{Z}}

\newcommand{\Hil}{\mathcal{H}}
\newcommand{\K}{\mathcal{K}}
\newcommand{\ip}[2]{\left( \left. #1 \, \right| \, #2 \right)}

\newcommand{\bip}[2]{\bigl( \bigl. #1 \, \bigr| \, #2 \bigr)}
\newcommand{\hip}[2]{\langle \! \langle #1, #2 \rangle \! \rangle}
\newcommand{\ess}{{\rm{ess}}}

\newcommand{\norm}[1]{{\left\| #1 \right\|}}

\newcommand{\Bignorm}[1]{{\Bigl\| #1 \Bigr\|}}

\newcommand{\supp}{\operatorname{supp}}
\newcommand{\Prim}{\operatorname{Prim}}

\newcommand{\id}{{\rm{id}}}
\newcommand{\spa}{\operatorname{span}}
\newcommand{\image}{\operatorname{im}}

\newcommand{\omax}{\otimes_{\rm{max}}}
\newcommand{\ohat}{\mathbin{\hat{\otimes}}}
\newcommand{\osig}{\otimes_\sigma}

\newcommand{\go}{{G^{(0)}}}

\newcommand{\iso}{\operatorname{Iso}}
\newcommand{\modu}{\Delta(\gamma)^{-1/2}}
\DeclareMathOperator{\Ind}{Ind}
\newcommand{\gcra}{\Gamma_c(G, r^*\A)}

\newtheorem{prop}{Proposition}[section]
\newtheorem{thm}[prop]{Theorem}
\newtheorem{cor}[prop]{Corollary}
\newtheorem{lem}[prop]{Lemma}
\theoremstyle{definition}
\newtheorem{defn}[prop]{Definition}

\newtheorem{rem}[prop]{Remark}

\newlist{thmenum}{enumerate}{10}
\setlist[thmenum,1]{label=\textnormal{(\alph*)}}
\setlist[thmenum,2]{label=\textnormal{(\roman*)}}

\newlist{altenum}{enumerate}{10}
\setlist[altenum,1]{label=\textnormal{(\roman*)}}
\setlist[altenum,2]{label=\textnormal{(\alph*)}}

\title{Nuclearity and Exactness for Groupoid Crossed Products}
\author{Scott M. LaLonde}
\address{Department of Mathematics, Dartmouth College, Hanover, NH 03755-3551}
\email{scott.m.lalonde.gr@dartmouth.edu}
\keywords{Groupoid crossed product, $C_0(X)$-algebra, nuclearity, exactness, exact groupoid.}
\subjclass[2010]{46L55, 46L06}

\begin{document}

\begin{abstract}
	Let $(\A, G, \alpha)$ be a groupoid dynamical system. We show that if $G$ is assumed to be measurewise amenable and the section 
	algebra $A = \Gamma_0(\go, \A)$ is nuclear, then the associated groupoid crossed product is also nuclear. This generalizes an earlier 
	result of Green for crossed products by locally compact groups. We also extend a related result of Kirchberg to groupoids. In particular, 
	if $A$ is exact and $G$ is amenable, then we show that $\A \rtimes G$ is exact.
\end{abstract}

\maketitle

\section{Introduction}
It has been known for quite some time that the amenability of a locally compact group is intimately connected with the nuclearity of the 
$C^*$-algebras associated to it. For example, Guichardet observed in \cite{guichardet} that if $G$ is an amenable group, then the group 
$C^*$-algebra $C^*(G)$ is nuclear. Of course this is now subsumed by the well-known fact that the class of nuclear $C^*$-algebras is stable 
under crossed products by amenable groups. This fact seems to have been first proved by Green in \cite{green}. 

Just as questions of amenability and nuclearity are often tightly wound together, the exactness of a group $C^*$-algebra or crossed product 
depends greatly on the properties of the underlying group. For example, if a locally compact group $G$ is exact in the sense of Kirchberg and 
Wassermann \cite{kw99}, then its reduced group $C^*$-algebra is exact. Since amenable groups are exact, $C^*(G)$ is exact when $G$ is 
amenable. This does not hold for arbitrary exact groups---Choi showed \cite{choi} that $C_r^*(\mathbb{F}_2)$ embeds into the 
Cuntz algebra $\mathcal{O}_2$, and is therefore exact. It follows that $\mathbb{F}_2$ is exact by \cite[Thm. 5.2]{kw99}, but the full group 
$C^*$-algebra $C^*(\mathbb{F}_2)$ is not exact \cite[Cor. 3.7]{wassermann}. More generally, this example shows that the crossed product 
of an exact $C^*$-algebra by an exact group need not be exact. However, the corresponding statement holds for the \emph{reduced} crossed 
product. It is this fact that lies at the heart of \cite[Prop. 7.1(v)]{kirchberg}, where Kirchberg shows that the crossed product of an exact 
$C^*$-algebra by an amenable group is exact.

Much attention has been given in recent years to the $C^*$-algebras that arise from locally compact groupoids. Operator algebraists have 
studied groupoid $C^*$-algebras, and more recently, groupoid crossed products. It is natural to ask whether the aforementioned results for 
group $C^*$-algebras and crossed products carry over to the groupoid setting. At least one of them is already known to generalize---it is 
shown in \cite[Cor. 6.2.14]{ananth-renault} that if $G$ is a measurewise amenable groupoid, then $C^*(G)$ is nuclear. In this paper, we 
extend Green's result to groupoid crossed products by showing that crossed products of nuclear $C^*$-algebras by measurewise amenable 
groupoids are nuclear. We also show that the crossed product of an exact $C^*$-algebra by an amenable groupoid is again exact, thus 
extending Kirchberg's result.

The structure of this paper is as follows. In Section 2 we give an overview of groupoid crossed products and their representations. 
In Section 3 we outline some technical results regarding ideals and representations of $C_0(X)$-algebras. We discuss tensor product 
dynamical systems in Section 5, and prove the first of two results on the relationship between tensor products and crossed products. With 
Section 6 comes the proof of the nuclearity theorem, and we take up exactness in Section 7.

Throughout we assume that all groupoids and topological spaces are second countable and all representations are nondegenerate unless 
otherwise specified. If $A$ is a $C^*$-algebra, then $M(A)$ denotes its multiplier algebra. Finally, we will frequently make reference to 
$C_0(X)$-algebras, in the sense of \cite[Appendix C]{TFB2}, and their associated upper semicontinuous $C^*$-bundles. We will always 
denote a $C^*$-bundle with a script letter (with the exception of $\Hil$, which is reserved for Hilbert spaces), and the corresponding Roman 
letter will represent its section algebra.

\section{Groupoid Crossed Products}
We begin with a brief overview of groupoid crossed products; more detailed treatments can be found in \cite{geoff} or \cite{mw08}. Throughout 
$G$ will denote a locally compact Hausdorff groupoid. We write $\go$ for the \emph{unit space} of $G$, and $r, s : G \to \go$ denote the
range and source maps, respectively. For $u \in \go$ we write $G^u := r^{-1}(\{u\})$ and $G_u := s^{-1}(\{u\})$. We also assume that $G$ is 
endowed with a continuous Haar system $\lambda = \{\lambda^u\}_{u \in \go}$. 

As with the classical case, groupoid crossed products are built out of $C^*$-dynamical systems. Since groupoids naturally act on fibered 
objects, our dynamical systems will not involve $C^*$-algebras \emph{per se}, but upper semicontinuous $C^*$-bundles over $\go$. 

\begin{defn}
	Let $\A$ be an upper semicontinuous bundle over $\go$. An \emph{action} of $G$ on $\A$ is a family $\alpha = 
	\{ \alpha_\gamma\}_{\gamma \in G}$, where:
	\begin{thmenum}
		\item $\alpha_\gamma : \A_{s(\gamma)} \to \A_{r(\gamma)}$ is an isomorphism for all $\gamma \in G$,
		\item if $(\gamma, \eta) \in G^{(2)}$, then $\alpha_{\gamma\eta} = \alpha_\gamma \circ \alpha_\eta$, and
		\item the assignment $(\gamma, a) \mapsto \gamma \cdot a = \alpha_\gamma(a)$ is continuous from $G * \A \to \A$.
	\end{thmenum}
\end{defn}

\begin{defn}
	A \emph{groupoid dynamical system} is a triple $(\A, G, \alpha)$, where $\A$ is an upper semicontinuous $C^*$-bundle and $\alpha$
	is an action of $G$ on $\A$. We say that $(\A, G, \alpha)$ is \emph{separable} if $A$ is separable and $G$ is second countable.
\end{defn}

\begin{rem}
	We assume throughout that all dynamical systems are separable, since many standard tools (e.g. Renault's Disintegration Theorem
	\cite[Thm. 7.12]{mw08}) are not available in the nonseparable case.
\end{rem}

Given a dynamical system $(\A, G, \alpha)$, the associated crossed product is built from a certain algebra of sections. In particular, 
we consider the space $\Gamma_c(G, r^*\A)$ of continuous compactly supported sections, which becomes a $*$-algebra with respect to 
the product
\[
	f*g(\gamma) = \int_G f(\eta) \alpha_\eta \bigl( g(\eta^{-1}\gamma) \bigr) \, d\lambda^{r(\gamma)}(\eta)
\]
and involution
\[
	f^*(\gamma) = \alpha_\gamma \bigl( f(\gamma^{-1})^* \bigr)
\]
One can verify that these operations are continuous with respect to the inductive limit topology which makes $\Gamma_c(G, r^*\A)$ into a 
topological $*$-algebra.

We can equip $\Gamma_c(G, r^*\A)$ with a norm in the following manner. A \emph{representation} of $\Gamma_c(G, r^*\A)$ on a Hilbert 
space $\Hil$ is a $*$-homomorphism $\pi : \Gamma_c(G, r^*\A) \to B(\Hil)$ that is continuous in the inductive limit topology. (A net $\{f_i\}$ 
converges to $f$ in the inductive limit topology if $f_i \to f$ uniformly and the sets $\supp(f_i)$ are eventually contained in a fixed compact set 
$K$.) For $f \in \Gamma_c(G, r^*\A)$, define
\[
	\norm{f} = \sup \bigl\{ \norm{\pi(f)} : \pi \text{ is a representation of } \Gamma_c(G, r^*\A) \bigr\},
\]
which we call the \emph{universal norm}. The completion of $\Gamma_c(G, r^*\A)$ with respect to the universal norm is called the \emph{(full) 
crossed product of $\A$ by $G$}, denoted $\A \rtimes_\alpha G$.

There is another way of obtaining the universal norm on $\Gamma_c(G, r^*\A)$, to which we will need to appeal later. As with group dynamical 
systems, there is a notion of \emph{covariant representations} for groupoid crossed products. However, they are quite technical, and we will use 
them only when necessary. If $\go * \frah$ is an analytic Borel Hilbert bundle, we define the \emph{isomorphism groupoid} to be
\[
	\iso(\go * \frah) = \bigl\{ (u, T, v) : T : \Hil(v) \to \Hil(u) \text{ is a unitary} \bigr\}.
\]
This is a groupoid under the operations
\[
	(u, T, v)(v, S, w) = (u, TS, w), \quad (u, T, v)^{-1} = (v, T^{-1}, u),
\]
and it is endowed with a natural Borel structure induced from the Borel sections of $\go * \frah$ \cite[\S F.6]{TFB2}. Thus $\iso(\go*\frah)$ is a 
Borel groupoid. A \emph{unitary representation} of $G$ is just a Borel groupoid homomorphism $U : G \to \iso(\go*\frah)$.

Now suppose $\mu$ is a \emph{quasi-invariant} Radon measure on $\go$, meaning that the induced measures $\nu = \mu \circ \lambda$ 
and $\nu^{-1} = \mu \circ \lambda^{-1}$ are equivalent. The Radon-Nikodym derivative $d\nu/d\nu^{-1}$ is denoted by $\Delta$ and called 
the \emph{modular function}. Let $\Hil$ denote the direct integral $L^2(\go*\frah, \mu)$, and suppose $\pi : A \to B(\Hil)$ is a $C_0(\go)$-linear 
representation. We will frequently use the fact that a $C_0(\go)$-linear representation can be decomposed into representations of the fibers 
of $\A$:
\[
	\pi = \int_{\go}^\oplus \pi_u \, d\mu(u), \, \, \pi_u : \A_u \to B(\Hil(u)),
\]
where the $\pi_u$ are $\mu$-a.e. nondegenerate and unique up to a null set \cite[Prop. 3.99]{geoff}. Finally, we say that the pair 
$(\pi, U)$ is \emph{covariant} if there is a $\nu$-null set $N \subset G$ such that for all $\gamma \not\in N$,
\[
	U_\gamma \pi_{s(\gamma)}(a) = \pi_{r(\gamma)}(\alpha_\gamma(a)) U_\gamma
\]
for all $a \in \A_{s(\gamma)}$. Given such a covariant representation, there is an associated representation of $\A \rtimes_\alpha G$, called
the \emph{integrated form} of $(\pi, U)$: for $f \in \Gamma_c(G, r^*\A)$, $h \in \Hil$, and $u \in \go$, we have
\[
	\pi \rtimes U(f)h(u) = \int_G \pi_u(f(\gamma)) U_\gamma f(s(\gamma)) \modu \, d\lambda^u(\gamma).
\]
Conversely, it is a consequence of Renault's Disintegration Theorem \cite[Thm. 7.12]{mw08} that every representation of $\A \rtimes_\alpha G$ 
is equivalent to the integrated form of a covariant representation. Consequently,
\[
	\norm{f} = \sup \bigl\{ \norm{\pi \rtimes U(f)} : (\pi, U) \text{ is a covariant representation of } (\A, G, \alpha) \bigr\}.
\]

Finally, we will need to invest heavily in induced representations of groupoid crossed products. There are several ways of viewing such 
representations, including a covariant ``left regular representation.'' However, we will usually opt for the avatar described in 
\cite[\S 2]{goehle2010} (or in \cite[\S 4.1]{sims-williams2013} for Fell bundles), which relies on Rieffel induction. Suppose $(\A, G, \alpha)$ 
is a groupoid dynamical system, and let $\pi : A \to B(\Hil)$ be a representation of the section algebra $A$. Then we can form a representation 
$\Ind \pi$ of $\A \rtimes_\alpha G$ as follows: the space of sections $\Zz_0 = \Gamma_c(G, s^*\A)$ is a right pre-Hilbert $A$-module with 
respect to the action
\[
	(z \cdot a)(\gamma) = z(\gamma) a(s(\gamma)), \quad z \in \Zz_0, a \in A, \gamma \in G
\]
and the $A$-valued inner product
\[
	\hip{z}{w}_A(u) = \int_G z(\xi)^*w(\xi) \, d\lambda_u(\xi), \quad z, w \in \Zz_0, u \in \go.
\]
We let $\Zz$ denote the completion of $\Zz_0$ with respect to the norm induced by $\hip{\cdot}{\cdot}_A$. Then $\A \rtimes_\alpha G$ 
acts on $\Zz$ by adjointable operators: for $f \in \Gamma_c(G, r^*\A)$ and $z \in \Zz_0$, 
\[
	f \cdot z(\gamma) = \int_G \alpha_\gamma^{-1}(f(\eta)) z(\eta^{-1}\gamma) \, d\lambda^{r(\gamma)}(\eta).
\]
We can then use Rieffel induction to construct the induced representation $\Ind \rho$. We equip $\Zz \odot \Hil$ with the inner product
characterized by
\[
	\ip{z \otimes h}{w \otimes k} = \ip{\pi\bigl( \hip{w}{z}_A \bigr) h}{k},
\]	
and we denote the completion by $\Zz \otimes_A \Hil$. Then $\Ind \pi$ acts on $\Zz \otimes_A \Hil$ by
\[
	\Ind \pi(f)(z \otimes h) = f \cdot z \otimes h
\]
for $f \in \Gamma_c(G, r^*\A)$, $z \in \Zz_0$, and $h \in \Hil$. If we take $\pi$ to be faithful, then we can define the \emph{reduced norm} on 
$\Gamma_c(G, r^*\A)$:
\[
	\norm{f}_r = \norm{\Ind \pi(f)}.
\]
The resulting completion is the \emph{reduced crossed product}, denoted by $\A \rtimes_{\alpha, r} G$.

\begin{rem}
	Some authors (such as \cite{sims-williams2013}) define the reduced norm instead as $\norm{f}_r = \sup \norm{\Ind \pi}$, where $\pi$ 
	ranges over all representations of $A$. Since induction preserves weak containment, this definition is equivalent to the one given above. In
	fact, it suffices to only consider representations lifted from the fibers of $A$. For each $u \in \go$, let $\rho_u$ be a faithful 
	representation of $A(u)$, and let $\pi_u$ denote its lift to $A$. Then $\bigcap \ker \pi_u = \{0\}$, so $\bigcap \ker \pi_u = \ker \pi$ for 
	any faithful representation $\pi$ of $A$. Consequently,
	\[
		\norm{f}_r = \sup_{u \in \go} \norm{\Ind \pi_u(f)}.
	\]
\end{rem}

\begin{rem}
We will occasionally need another flavor of induced representation. Let $(\A, G, \alpha)$ be a separable groupoid dynamical system, 
and let $\pi$ be a nondegenerate representation of $A$. Using \cite[Thm. 8.3.2]{dixmier}, assume that $\pi$ is a $C_0(\go)$-linear 
representation on a Borel Hilbert bundle $\go * \frah$ with associated finite Borel measure $\mu$. Let $\nu^{-1} = \int_{\go} \lambda_u \, d\mu$, 
and form the pullback bundle $G *_s \frah = s^*(\go*\frah)$. For $h \in L^2(G*_s \frah, \nu^{-1})$, define
\begin{equation}
\label{eq:IntegratedLR}
	L_\pi(f) h (\gamma) = \int_G \pi_{s(\gamma)} \bigl( \alpha_{\gamma}^{-1}(f(\eta)) \bigr) h(\eta^{-1} \gamma) \, 
		d\lambda^{r(\gamma)}(\eta).
\end{equation}
Then $L_\pi$ defines an $I$-norm decreasing representation of $\gcra$, which extends to a representation of $\A \rtimes_\alpha G$. 
Furthermore, it is shown in \cite[Lem. 2]{sims-williams2013} that the map $U : \Zme \otimes_A \Hil \to L^2(G *_s \frah, \nu^{-1})$ characterized 
by
	\[
		U(z \otimes h)(\gamma) = \pi_{s(\gamma)} \bigl( z(\gamma) \bigr) h(s(\gamma))
	\]
is a unitary intertwining $L_\pi$ and $\Ind \pi$.
\end{rem}

\section{Preliminaries on $C_0(X)$-algebras}
Our first goal is to prove an analogue of Green's theorem \cite[Prop. 14]{green} for groupoid crossed products. The proof relies 
on \cite[Lem. 2.75]{TFB2}, so an obvious first step would be to extend this result to groupoid dynamical systems. This in turn requires the
following: given a groupoid dynamical system $(\A, G, \alpha)$ and a $C^*$-algebra $B$, we need to define new dynamical systems 
$(\A \omax B, G, \alpha \otimes \id)$ and $(\A \osig B, G, \alpha \otimes \id)$ by allowing $G$ to act trivially on $B$. For this to even make 
sense, we need to know that $A \omax B$ and $A \osig B$ are $C_0 \bigl(G^{(0)} \bigr)$-algebras, and we need to understand the structure 
of the corresponding $C^*$-bundles $\A \omax B$ and $\A \osig B$.

\subsection{Tensor Products of $C_0(X)$-algebras}
Throughout this section, $X$ will denote a locally compact Hausdorff space. We begin by recalling some results of Kirchberg and 
Wassermann on tensor product bundles \cite{kw95}. Suppose $A$ is a $C_0(X)$-algebra with associated upper semicontinuous 
$C^*$-bundle $\A$. Let $B$ be a fixed $C^*$-algebra, and let $\norm{\cdot}_\nu$ be a $C^*$-norm on the algebraic tensor product 
$A \odot B$. Given $x \in X$, let $C_{0,x}(X) \subset C_0(X)$ denote the ideal of functions vanishing at $x$, and put 
$I_x = C_{0,x}(X) \cdot A$. Then $I_x \otimes_\nu B$ is an ideal of $A \otimes_\nu B$, and the quotient $(A \otimes_\nu B)/(I_x \otimes_\nu B)$ is
isomorphic to $\A_x \otimes_{\nu(x)} B$ for some $C^*$-norm $\norm{\cdot}_{\nu(x)}$ on $\A_x \odot B$ \cite[Prop. 3.7.2]{brownozawa}. The 
bundle $\A \otimes_{\nu} B$ with fibers given by $(\A \otimes_\nu B)_x = \A_x \otimes_{\nu(x)} B$ is an upper semicontinuous $C^*$-bundle
by \cite[Lem. 2.4]{kw95}.

Unfortunately, the norms $\norm{\cdot}_{\nu(x)}$ depend on $x \in X$, and we have no control over how they vary in general. We are 
predominantly interested in the cases where $\norm{\cdot}_\nu$ is either the maximal or minimal norm on $A \odot B$. Things are especially
nice when working with the maximal tensor product, since the sequence
\[
	\xymatrix{
		0 \ar[r] & I_x \omax B \ar[r] & A \omax B \ar[r] & \A_x \omax B \ar[r] & 0
	}
\]
is always exact \cite[Prop. B.30]{TFB1}. This ensures that
\[
	(\A \omax B)_x = \A_x \omax B
\]
for all $x \in X$. It is natural to ask whether the same sort of thing holds for $\A \osig B$. The answer hinges upon the exactness of the 
sequence
\begin{equation}
\label{eq:exactfiberseq}
\xymatrix{
	0 \ar[r] &  I_x \osig B \ar[r] & A \osig B \ar[r] & \A_x \osig B \ar[r] & 0,
}
\end{equation}
which may fail in general. However, it is observed in the proof of \cite[Cor. 2.8]{kw95}, and we prove below, that \eqref{eq:exactfiberseq} is
exact when either $A$ or $B$ is an exact $C^*$-algebra.

\begin{prop}
\label{prop:minimalfiber}
	Let $A$ be a separable $C_0(X)$-algebra and $B$ a $C^*$-algebra. If either $A$ or $B$ is exact, then $(\A \osig B)_x \cong 
	\A_x \osig B$.
\end{prop}
\begin{proof}
	If $B$ is exact, then \eqref{eq:exactfiberseq} is clearly exact. On the other hand, if $A$ is exact, then it is subnuclear by 
	\cite[Thm. IV.3.4.18]{blackadar}. As a consequence of \cite[Prop. IV.3.4.7]{blackadar}, nuclear $C^*$-algebras 
	have Property C of Archbold and Batty \cite{archboldbatty}, so it follows from \cite[Cor. IV.3.4.10]{blackadar} that $A$ has Property C. 
	The exactness of \eqref{eq:exactfiberseq} then follows from \cite[Cor. IV.3.4.11]{blackadar}.
\end{proof}
 
The following two propositions deal with representations of maximal tensor products involving $C_0(X)$-algebras. In both $X$ denotes a 
second countable locally compact Hausdorff space.

\begin{prop}
\label{prop:linrep}
	Let $A$ be a $C_0(X)$-algebra and $B$ a $C^*$-algebra. Suppose $X*\frah$ is an analytic Borel Hilbert bundle over $X$, $\mu$ is a Borel 
	measure on $X$, and $\pi = \pi_A \omax \pi_B$ is a $C_0(X)$-linear representation of $A \omax B$ on $L^2(X*\frah, \mu)$. Then $\pi_A$ is 
	$C_0(X)$-linear.
	
	Conversely, suppose that $\pi_A$ and $\pi_B$ are representations of $A$ and $B$, respectively, on $L^2(X*\frah, \mu)$ with commuting 
	ranges, and that $\pi_A$ is $C_0(X)$-linear. Then the representation $\pi_A \omax \pi_B$ of $A \omax B$ on $L^2(X*\frah, \mu)$ is also 
	$C_0(X)$-linear.
\end{prop}
\begin{proof}
	For $f \in C_0(X)$, $a \in A$, and $b \in B$, we have
	\[
		\pi \bigl( f \cdot (a \otimes b) \bigr) = \pi_A \omax \pi_B \bigl( (f \cdot a) \otimes b \bigr) = \pi_A (f \cdot a) \pi_B(b).
	\]
	But since $\pi$ is $C_0(X)$-linear, we also have
	\[
		\pi \bigl( f \cdot (a \otimes b) \bigr) = T_f \pi(a \otimes b) = T_f \pi_A(a) \pi_B(b),
	\]
	where $T_f \in B(L^2(X*\frah, \mu))$ is given by pointwise multiplication by $f$. Since $\pi_B$ is nondegenerate, it follows that 
	$\pi_A (f \cdot a) = T_f \pi_A(a)$, so $\pi_A$ is $C_0(X)$-linear. The converse is similar. 
\end{proof}

\begin{prop}
\label{prop:linrep2}
	Let $A$ be a separable $C_0(X)$-algebra and $B$ a separable $C^*$-algebra. If $X*\frah$ is an analytic Borel Hilbert bundle, $\mu$ 
	is a finite Borel measure on $X$, and $\pi_A \omax \pi_B : A \omax B \to B(L^2(X*\frah, \mu))$ is a $C_0(X)$-linear representation, 
	then $\pi_B$ is decomposable. Moreover, if $\{ \pi_{A,x} \}_{x \in X}$ is a decomposition of $\pi_A$ into representations of the fibers 
	$\A_x$, and $\{ \pi_{B, x} \}_{x \in X}$ is a decomposition of $\pi_B$, then 
	\begin{equation}
	\label{eq:tensordecomp}
		\pi_A \omax \pi_B(t) = \int_X^\oplus \pi_{A, x} \omax \pi_{B,x} \bigl( t(x) \bigr) \, d\mu(x)
	\end{equation}
	for all $t \in A \omax B$.
\end{prop}
\begin{proof}
	Since $T \in L^2(X*\frah, \mu)$ is decomposable if and only if it commutes with the set $\Delta(X*\frah, \mu)$ of diagonal 
	operators \cite[Thm. F.21]{TFB2}, we need to check that $\pi_B(b) \in ( \Delta(X *\frah, \mu) )'$ for all $b \in B$. Fix $b \in 
	B$, and let $f \in C_0(X)$ and $a \in A$. Then 
	\[
		\pi_A(a) \pi_B(b) T_f = \pi_A \omax \pi_B(a \otimes b) T_f = T_f \bigl( \pi_A \omax \pi_B (a \otimes b) \bigr)
	\]
	since $\pi_A \omax \pi_B$ is $C_0(X)$-linear, hence decomposable. Continuing, we get
	\[
		T_f \bigl( \pi_A \omax \pi_B (a \otimes b) \bigr) = T_f \pi_A(a) \pi_B(b) = \pi_A(a) T_f \pi_B(b),
	\]
	since $\pi_A$ is decomposable. Now since $\pi_A$ is nondegenerate, if follows that $\pi_B(b) T_f = T_f \pi_B(b)$, so $\pi_B(b) \in 
	(\Delta(X*\frah, \mu))'$ for all $b \in B$. Thus the operators $\pi_B(b)$ are all decomposable, so there is a family of maps 
	$\pi_{B,x} : B \to B(\Hil(x))$ such that $\pi_{B,x}(b) = \pi_{B}(b)(x)$. The proof of \cite[Prop. 3.99]{geoff} implies that modifying  
	the $\pi_{B,x}$ on a $\mu$-null set will yield nondegenerate representations of $B$ for $\mu$-almost all $x$.
	
	For the second assertion, let $\{\pi_{B, x}\}_{x \in X}$ denote a decomposition of $\pi_B$, and suppose $\{\pi_{A,x}\}_{x \in X}$ is a 
	decomposition of $\pi_A$ into representations of the fibers $\A_x$. Then for each $x \in X$, $\pi_{A, x} \omax \pi_{B, x}$ is a 
	representation of $(\A \omax B)_x = \A_x \omax B$, and $\pi_{A, x} \omax \pi_{B, x}$ is nondegenerate for $\mu$-almost all $x$ since 
	$\pi_{A,x}$ and $\pi_{B,x}$ are. Now let $\{\pi_x\}_{x \in X}$ be a decomposition of $\pi_A \omax \pi_B$, so
	\[
		\pi_A \omax \pi_B(t)h(x) = \pi_x \bigl( t(x) \bigr) h(x)
	\]
	for all $x \in X$ and $h \in L^2(X*\frah, \mu)$. Then we have
	\begin{align*}
		\pi_A \omax \pi_B(a \otimes b)h(x) &= \pi_A(a) \pi_B(b) h(x) \\
			&= \pi_{A, x} \bigl( a(x) \bigr) \pi_{B,x}(b) h(x) \\
			&= \pi_{A,x} \omax \pi_{B,x} \bigl( a(x) \otimes b \bigr) h(x)
	\end{align*}
	for all $a \in A$, $b \in B$, and $h \in L^2(X*\frah, \mu)$. Since the $\pi_x$ are unique up to a $\mu$-null set, it follows that
	$\pi_x = \pi_{A,x} \omax \pi_{B,x}$ $\mu$-almost everywhere. 
\end{proof}

\subsection{Pullbacks and $C_0(X)$-linear Homomorphisms}
Let $Y$ be a locally compact Hausdorff space, and suppose $\tau : Y \to X$ is continuous. If $A$ is a $C_0(X)$-algebra with associated 
upper semicontinuous bundle $p: \A \to X$, recall that the \emph{pullback bundle} $\tau^*\A$ over $Y$ is the bundle with total space
\[
	\tau^*\A = \left\{ (y, a) \in Y \times \A : p(a) = \tau(y) \right\}
\]
and structure map $q : \tau^*\A \to Y$ given by $q(y, a) = y$. The \emph{pullback $C^*$-algebra} is 
\[
	\tau^*A = \Gamma_0(Y, \tau^*\A).
\]
One can verify \cite[Prop. 3.34]{geoff} that $\tau^*\A$ is an upper semicontinuous $C^*$-bundle, so $\tau^*A$ can be viewed as a 
$C_0(Y)$-algebra. 

It is often helpful to approximate sections of a pullback bundle with ``elementary tensors.'' Suppose $X$ and $Y$ are locally compact 
Hausdorff spaces, $A$ is a $C_0(X)$-algebra with associated upper-semicontinuous bundle $\A$, and $\tau : Y \to X$ is continuous. Given 
$z \in C_c(Y)$ and $a \in A$, we define $z \otimes a \in \Gamma_c(Y, \tau^*\A)$ by
\[
	z \otimes a(y) = z(y)a(\tau(y)).
\]
Then the set
\[
	C_c(Y) \odot A = \spa \{ z \otimes a : z \in C_c(Y), a \in A\}
\]
is dense in $\Gamma_c(Y, \tau^*\A)$ with respect to the inductive limit topology \cite[Cor. 3.45]{geoff}. We prove a similar fact for tensor 
product bundles. Let $B$ be a $C^*$-algebra. Given $f \in \Gamma_c(Y, \tau^*\A)$ and $b \in B$, define $f \ohat b \in 
\Gamma_c(Y, \tau^*(\A \omax B))$ by
\[
	f \ohat b(y) = f(y) \otimes b
\]
using the canonical identification of $\tau^*(\A \omax B)_y$ with $\A_{\tau(y)} \omax B$. In this way, $\Gamma_c(Y, \tau^*\A) \odot B$ can be 
viewed naturally as a subspace of $\Gamma_c(Y, \tau^*(\A \omax B))$ via the embedding $f \otimes b \mapsto f \ohat b$.

\begin{prop}
\label{prop:tensorILT}
	The set
	\[
		\Gamma_c(Y, \tau^*\A) \odot B = \spa \{ f \ohat b : f \in \Gamma_c(Y, \tau^*\A), b \in B \}
	\]
	is dense in $\Gamma_c(Y, \tau^*(\A \omax B))$ with respect to the inductive limit topology.
\end{prop}
\begin{proof}
	Since $C_c(Y) \odot (A \omax B)$ is dense in $\Gamma_c(Y, \tau^*(\A \omax B)$ with respect to the inductive limit topology, it 
	suffices to show that we can approximate any element of $C_c(Y) \odot (A \omax B)$ with elements of $\Gamma_c(Y, \tau^*\A) \odot B$.
	Let $z \in C_c(Y)$ and $t \in A \omax B$, and choose a sequence $\{t_i\} \subset A \odot B$ that converges to $t$ in 
	$A \omax B$. Given $\varepsilon > 0$, choose $i$ large enough such that $\norm{t_i - t}_{\rm{max}} < \varepsilon/\norm{z}_\infty$. 
	Then
	\[
		\norm{z \otimes t_i(y) - z \otimes t(y)} = \norm{z(y)t_i(\tau(y)) - z(y)t(\tau(y))} \leq \norm{z}_\infty \norm{t_i - t}_{\rm{max}} < \varepsilon
	\]
	for any $y \in Y$, so $z \otimes t_i \to z \otimes t$ uniformly. In addition, we clearly have $\supp(z \otimes t_i) \subset \supp(z)$ for all $i$, 
	so $z \otimes t_i \to z \otimes t$ in the inductive limit topology. Now fix $i$ and write $t_i = \sum a_{ij} \otimes b_{ij}$. Then for each 
	$y \in Y$,
	\begin{align*}
		z \otimes t_i(y) &=  z(y) \sum (a_{ij} \otimes b_{ij})(\tau(y)) \\
			&= \sum z(y)a_{ij}(\tau(y)) \otimes b_{ij} \\
			&= \sum (z \otimes a_{ij})(\tau(y)) \otimes b_{ij} \\
			&= \bigl( \sum (z \otimes a_{ij}) \ohat b_{ij} \bigr)(y),
	\end{align*}
	so $z \otimes t_i = \sum (z \otimes a_{ij}) \ohat b_{ij} \in \Gamma_c(Y, \tau^*\A) \odot B$. The result then follows.
\end{proof}

We will often work with homomorphisms between $C_0(X)$-algebras. Recall the following definition.

\begin{defn}
	Let $A$ and $B$ be $C_0(X)$-algebras. A homomorphism $\varphi : A \to B$ is said to be \emph{$C_0(X)$-linear} if
	\[
		\varphi(f \cdot a) = f \cdot \varphi(a)
	\]
	for all $f \in C_0(X)$ and $a \in A$.
\end{defn}

\begin{rem}
\label{rem:bundlecorrespondence}
	Given two $C_0(X)$-algebras $A$ and $B$ with associated
	bundles $\A$ and $\B$, there is a correspondence between $C_0(X)$-linear homomorphisms from $A$ to $B$ and $C^*$-bundle
	morphisms from $\A$ to $\B$. If $\varphi : A \to B$ is a $C_0(X)$-linear homomorphism, then for each $x \in X$ there is
	a homomorphism $\varphi_x : \A_x \to \B_x$ satisfying
	\begin{equation}
	\label{eq:c0x-1}
		\varphi(a)(x) = \varphi_x(a(x))
	\end{equation}
	for all $a \in A$. Moreover, the $\varphi_x$ vary continuously with $x$, and thus glue together to yield a bundle morphism $\hat{\varphi} : 
	\A \to \B$. Conversely, if $\psi : \A \to \B$ is a $C^*$-bundle homomorphism, there is a $C_0(X)$-linear homomorphism 
	$\check{\psi} : A \to B$ given by
	\[
		\check{\psi}(a)(x) = \psi(a(x))
	\]
	for all $a \in A$. This is the content of \cite[Prop. 3.20]{geoff}, and it is discussed in \cite[\S 3]{mw08}.
\end{rem}

Our next goal is to extend Remark \ref{rem:bundlecorrespondence} to pullbacks of $C_0(X)$-linear homomorphisms. We will need to use the
following fact from \cite[Prop. 3.35]{geoff} and \cite[Rem. 3.5]{mw08}: a section $f : Y \to \tau^*\A$ is continuous if and only if there is a 
continuous function $\tilde{f} : Y \to \A$ such that $p \bigl( \tilde{f}(y) \bigr) = \tau(y)$ and $f(y) = \bigl( y, \tilde{f}(y) \bigr)$ for all $y \in Y$. 
Moreover, $\tilde{f}$ vanishes at infinity if and only if $f$ does, and $\tilde{f}$ is compactly supported if and only if $f$ is.

\begin{prop}
\label{prop:pullbackhom}
	Let $A$ and $B$ be $C_0(X)$-algebras with upper semicontinuous bundles $\A$ and $\B$, respectively, and suppose 
	$\varphi : A \to B$ is a $C_0(X)$-linear homomorphism. Let $\tau : Y \to X$ be a continuous map. Then there is a $C_0(Y)$-linear 
	homomorphism $\tau^*\varphi : \tau^*A \to \tau^*B$ given by
	\[
		\tau^*\varphi(f)(y) = \hat{\varphi}\bigl( \tilde{f}(y) \bigr)
	\]
	for all $f \in \tau^*A$ and $y \in Y$. Moreover, $\tau^*\varphi$ carries $\Gamma_c(Y, \tau^*\A)$ into $\Gamma_c(Y, \tau^*\B)$.
\end{prop}
\begin{proof}
	Define $\tau^*\varphi$ as above. We first need to check that if $f \in \Gamma_0(X, \tau^*\A)$, then $\tau^*\varphi(f) \in \Gamma_0(Y, \tau^*\B)
	$. We know that $\tilde{f}(y) \in 
	\A_{\tau(y)}$ for all $y \in Y$, so $\hat{\varphi} \bigl( \tilde{f}(y) \bigr) \in \B_{\tau(y)}$. Thus $\tau^*\varphi(f)(y) \in \tau^*\B_y$, and 
	$\tau^*\varphi(f)$ is a section of $\tau^*\B$. Since $y \mapsto \hat{\varphi}\bigl( \tilde{f}(y \bigr)$ from $Y \to \B$ is continuous, 
	$\tau^*\varphi(f)$ is continuous and vanishes at infinity since $\tilde{f}$ does. Thus 
	$\tau^*\varphi$ maps $\tau^*A$ into $\tau^*B$.
	
	The only thing left to check is that $\tau^*\varphi$ is a $C_0(Y)$-linear homomorphism. If $f, g \in \tau^*\A$, it is straightforward to 
	check that $(f+g)^\sim = \tilde{f} + \tilde{g}$, so
	\[
		\tau^*\varphi(f+g)(y) = \hat{\varphi} \bigl( \tilde{f}(y) \bigr) + \hat{\varphi} \bigl( \tilde{g}(y) \bigr) 
			= \tau^*\varphi(f)(y) + \tau^*\varphi(g)(y).
	\]
	A similar argument shows that $\tau^*\varphi$ respects multiplication, scalar multiplication, and adjoints, so $\tau^*\varphi$ is a 
	homomorphism. Finally, suppose $\sigma \in C_0(Y)$. Then it is straightforward to check that $(\sigma \cdot f)^\sim(y) = \sigma(y) 
	\tilde{f}(y)$ for all $y \in Y$, so
	\[
		\tau^*\varphi( \sigma \cdot f)(y) = \sigma(y) \hat{\varphi} \bigl( \tilde{f}(y) \bigr) = \bigl( \sigma \cdot \tau^*\varphi(f) \bigr)(y).
	\]
	Thus $\tau^*\varphi$ is $C_0(Y)$-linear, and we are done.
\end{proof}

In light of the natural identification of $(\tau^*\A)_y$ with $\A_{\tau(y)}$, we will generally suppress the tilde in Proposition 
\ref{prop:pullbackhom}, and simply write
\[
	\tau^*\varphi(f)(y) = \hat{\varphi}(f(y)) = \varphi_{\tau(y)}(f(y)).
\]
Since $\tau^*\varphi$ is $C_0(Y)$-linear, it admits a fiberwise decomposition as in Remark \ref{rem:bundlecorrespondence}. Moreover, the 
identifications $(\tau^*\A)_y = \A_{\tau(y)}$ and $(\tau^*\B)_y = \B_{\tau(y)}$ lend a concrete description of the homomorphism 
$(\tau^*\varphi)_y : (\tau^*\A)_y \to (\tau^*\B)_y$.

\begin{prop}
\label{prop:pullbackhom2}
	For all $y \in Y$ and $a \in (\tau^*\A)_y = \A_{\tau(y)}$, we have 
	\[
		(\tau^*\varphi)_y(a) = \varphi_{\tau(y)}(a).
	\]
\end{prop}
\begin{proof}
	If $y \in Y$ then $(\tau^*\varphi)_y$ is characterized by
	\[
		(\tau^*\varphi)_y(f(y)) = \tau^*\varphi(f)(y)
	\]
	for all $f \in \Gamma_0(Y, \tau^*\A)$. By definition, $\tau^*\varphi(f)(y) = \hat{\varphi}(f(y))$, where we view $f(y)$ as an element 
	of $(\tau^*\A)_y = \A_{\tau(y)}$. But $\hat{\varphi}(f(y)) = \varphi_{\tau(y)}(f(y))$, and the result follows.
\end{proof}

\subsection{Ideals and Quotients}
Exactness for groupoid crossed products naturally involves ideals and quotients of $C_0(X)$-algebras. It is a crucial fact that any ideal of a 
$C_0(X)$-algebra is again a $C_0(X)$-algebra, and similarly for quotients. Moreover, the structures of the associated bundles are fairly predictable. 
We begin with ideals.

\begin{prop}
\label{prop:C0Xideal}
	Let $A$ be a $C_0(X)$-algebra and $I$ an ideal of $A$. Then $I$ is also a $C_0(X)$-algebra. More precisely, if $\Phi_A : C_0(X) \to 
	ZM(A)$ implements the action of $C_0(X)$ on $A$, then $I$ is invariant under $\Phi_A(C_0(X))$ and the action of $C_0(X)$ on $I$ is 
	simply the restriction of the action on $A$.
\end{prop}
\begin{proof}
	Recall that $A$ is a $C_0(X)$-algebra if and only if there is a continuous map $\sigma_A : \Prim A \to X$ \cite[Thm. C.26]{TFB2} and that 
	there is a homeomorphism $\iota$ of $\Prim I$ onto the open subspace $\{ P \in \Prim A : I \not\subset P\}$ of $\Prim A$ by 
	\cite[Prop. A.27]{TFB1}. Therefore, $\sigma_I = \sigma_A \circ \iota$ is continuous from $\Prim I$ to $X$, so $I$ 
	is a $C_0(X)$-algebra.

	Now recall that $ZM(A) \cong C^b(\Prim A)$ by the Dauns-Hofmann theorem, and under this identification we have 
	$\Phi_A(f) = f \circ \sigma_A$ for all $f \in C_0(X)$. Similarly, if $\Phi_I : C_0(X) \to ZM(I) \cong C^b(\Prim I)$ implements the 
	$C_0(X)$-action on $I$, then 
	\[
		\Phi_I(f) = f \circ \sigma_I = f \circ (\sigma_A \circ \iota) = \Phi_A(f) \circ \iota.
	\]
	If $a \in A$ and $P$ is any primitive ideal of $A$, we let $a(P)$ denote the image of $a$ in the primitive quotient $A/P$. Then we have
	\begin{equation}
	\label{eq:c0Xaction}
		 \bigl( \Phi_A(f) \cdot a \bigr)(P) = \Phi_A(f)(P) a(P)
	\end{equation}
	for all $a \in A$ and $P \in \Prim A$. Now suppose that $a \in I$, and let $P$ be any primitive ideal that contains $I$. Then $a(P) = 0$, 
	so for any $f \in C_0(X)$, \eqref{eq:c0Xaction} implies
	\[
		 \bigl( \Phi_A(f) \cdot a \bigr)(P) = 0,
	\]
	and it follows that $\Phi_A(f) \cdot a \in P$. This is true for any primitive ideal containing $I$, so $\Phi_A(f) \cdot a \in I$. Thus $I$ 
	is invariant under the $C_0(X)$-action on $A$. Finally, if we identify $\Prim I$ with $\{ P \in \Prim A : I \not\subset P\}$ via $\iota$, then 
	for any $P \in \Prim A$ not containing $I$ we have
	\[
		\bigl( \Phi_I(f) \cdot a \bigr)(P) = \Phi_I(f)(P) a(P) = \Phi_A(f)(P) a(P),
	\]
	so the action of $C_0(X)$ on $I$ is simply the restriction of the action on $A$.
\end{proof}

We have now pinned down two of the characterizations of $C_0(X)$-algebras for ideals. However, we do not yet have a description of 
the upper semicontinuous $C^*$-bundle associated to $I$. The following is based on \cite[Prop. 3.3]{dana-marius}.

\begin{prop}
	Let $A$ be a $C_0(X)$-algebra with associated upper semicontinuous bundle $p : \A \to X$, and let $I$ be an ideal in $A$. For each 
	$x \in X$, let $q_x : A \to \A_x = A/J_x$ denote the quotient map onto the fiber over $x$, and define
	\[
		\I_x = q_x(I) = (I+J_x)/J_x.
	\]
	Then
	\[
		\I = \coprod_{x \in X} \I_x \quad \text{and} \quad p \vert_{\I} : \I \to X
	\]
	define an upper semicontinuous bundle over $X$ with $\Gamma_0(X, \I) \cong I$.
\end{prop}
\begin{proof}
	Since each $\I_x$ is an ideal in $\A_x$, $\I$ clearly includes into $\A$. The only thing that really needs to be checked is the openness
	of $p \vert_{\I}$. To do this, we'll use \cite[Prop. 1.15]{TFB2}. Let $\{x_i\}$ be a net in $X$ converging to $x \in X$, and fix $c \in \I_x$. 
	Use Cohen Factorization to write $c = ab$ for $a \in \A_x$ and $b \in \I_x$. Since $p : \A \to X$ is open, we can pass to a subnet, relabel, 
	and find a net $\{a_i\}$ in $\A$ such that $p(a_i) = x_i$ for all $i$ and $a_i \to a$. Now choose $\sigma \in I$ with $\sigma(x) = b$, and 
	put $b_i = \sigma(x_i)$ for all $i$. Then $a_i b_i \in \I_{x_i}$ for all $i$ and $a_i b_i \to a \cdot \sigma(x) = c$. It follows that 
	$p \vert_{\I}$ is open, and $p \vert_{\I} : \I \to X$ is an upper semicontinuous $C^*$-bundle.
	
	We have so far shown that $\I$ sits inside $\A$ as a subbundle. Since each fiber $\I_x$ is an ideal in $\A_x$, it is clear that 
	$\Gamma_0(X, \I)$ sits naturally inside $\Gamma_0(X, \A)$ as an ideal. It is then easy to see that $\Gamma_0(X, \I)$ can be identified 
	with $I$.
\end{proof}

The analogous results for quotients of $C_0(X)$-algebras are proven in a similar fashion. In fact, they are more or less taken care of in 
\cite[Lem. 1.3]{dana-marius}.

\begin{prop}[{\cite[Lem. 1.3]{dana-marius}}]
\label{prop:C0Xquotient}
	Let $A$ be a $C_0(X)$-algebra with the action implemented by the homomorphism $\Phi_A : C_0(X) \to ZM(A)$. If $I$ is an ideal of $A$, 
	then $A/I$ is a $C_0(X)$-algebra, and the action $\Phi_{A/I} : C_0(X) \to ZM(A/I)$ is characterized by
	\[
		\Phi_{A/I}(f) \cdot (a+I) = \bigl( \Phi_A(f) \cdot a \bigr) + I
	\]
	for all $f \in C_0(X)$ and $a \in A$. Moreover, if $\A/\I$ denotes the upper semicontinuous $C^*$-bundle associated to $A/I$, then 
	$(\A/\I)_x$ is isomorphic to $\A_x/\I_x$ for all $x \in X$.
\end{prop}

\section{Tensor Product Systems}
Let us return now to the situation where $(\A, G, \alpha)$ is a separable groupoid dynamical system and $B$ is a separable $C^*$-algebra.
The results of Section 3.1 guarantee that $A \omax B$ and $A \osig B$ are $C_0(X)$-algebras whenever $A$ is, and that the fibers of 
$\A \omax B$ and $\A \osig B$ are easy to understand. This will make it relatively straightforward to define actions of $G$ on $\A \omax B$ and 
$\A \osig B$. For $\A \omax B$, this amounts to defining a family $\alpha \otimes \id = \{\alpha_\gamma \omax \id \}_{\gamma \in G}$ of fiberwise 
isomorphisms. Under the identification $(\A \omax B)_u = \A_u \omax B$, we have
\[
	\alpha_\gamma \omax \id : \A_{s(\gamma)} \omax B \to \A_{r(\gamma)} \omax B,
\]
which is an isomorphism by \cite[Prop. B.30]{TFB1}. It is then be clear that
\[
	\alpha_{\gamma \eta} \omax \id = (\alpha_\gamma \circ \alpha_\eta) \omax \id = (\alpha_\gamma \omax \id) \circ 
		(\alpha_\eta \omax \id)
\]
whenever $(\gamma, \eta) \in G^{(2)}$. Similarly, if we assume that $A$ is exact then Proposition \ref{prop:minimalfiber} implies that $(\A 
\osig B)_u$ is identified with $\A_u \osig B$, and we get an isomorphism
\[
	\alpha_\gamma \otimes \id : \A_{s(\gamma)} \osig B \to \A_{r(\gamma)} \osig B
\]
by \cite[Prop. B.13]{TFB1}. It is again straightforward to see that 
\[
	\alpha_{\gamma \eta} \otimes \id = (\alpha_\gamma \otimes \id) \circ (\alpha_\eta \otimes \id)
\]
whenever $(\gamma, \eta) \in G^{(2)}$. Therefore, $G$ acts on both $\A \omax B$ and $\A \osig B$ in the obvious way, at least in a purely 
algebraic sense. It remains to check that $\gamma \cdot t = (\alpha_\gamma \omax \id)(t)$ defines a continuous action of $G$ on $\A \omax B$, 
and similarly for the action $\gamma \cdot t = (\alpha_\gamma \otimes \id)(t)$ on $\A \osig B$. With a little work, this follows from the fact that 
the action $\alpha$ of $G$ on $\A$ is continuous. Indeed, the proofs of the following theorems are restrictions of the one found on page 919 
of \cite{KMRW} (as noted in \cite[\S 4]{BG}).

\begin{thm}
\label{thm:tensoraction}
	Let $(\A, G, \alpha)$ be a separable groupoid dynamical system, and let $B$ be a separable $C^*$-algebra.
	\begin{enumerate}
		\item The assignment
	\[
		\gamma \cdot c = (\alpha_\gamma \omax \id)(c)
	\]
	defines a continuous action of $G$ on $\A \omax {B}$, so $(\A \omax B, G, \alpha \omax \id)$ is a separable groupoid dynamical system.
	\item If $A$ is exact, then
	\[
		\gamma \cdot c = (\alpha_\gamma \otimes \id)(c)
	\]
	defines a continuous action of $G$ on $\A \osig B$, and $(\A \osig B, G, \alpha \otimes \id)$ is a separable groupoid dynamical system.
	\end{enumerate}
\end{thm}
\begin{proof}
	If $A$ is exact, then Proposition \ref{prop:minimalfiber} guarantees that $(\A \osig B)_u \cong \A_u \osig B$ for all 
	$u \in \go$, and hence that $\alpha_\gamma \otimes \id$ is an isomorphism for all $\gamma \in G$. The continuity of the action 
	is then proven in exactly the same way as for $\A \omax B$. 
\end{proof}
 
We will end this section with a generalization of \cite[Lem. 2.75]{TFB2}. Recall that Proposition \ref{prop:tensorILT} guarantees there is a 
natural embedding of $\Gamma_c(G, r^*\A) \odot B$ into $\Gamma_c(G, r^*(\A \omax B))$, and that the image is dense with respect to the 
inductive limit topology. We now show that this embedding extends to an isomorphism of $(\A \rtimes_\alpha G) \omax B$ onto 
$(\A \omax B) \rtimes_{\alpha \otimes \id} G$.

\begin{thm}
\label{thm:exchange}
	Let $(\A, G, \alpha)$ be a separable dynamical system, and suppose $B$ is a separable $C^*$-algebra. There is a natural isomorphism
	\[
		\Phi : (\A \rtimes_\alpha G) \omax B \to (\A \omax B) \rtimes_{\alpha \otimes \id} G,
	\]
	given on elementary tensors by
	\[
		\Phi(f \otimes b) = f \ohat b,
	\]
	for $f \in \Gamma_c(G, r^*\A)$ and $b \in B$.
\end{thm}

We will give a proof based on the first part of Green's proof of \cite[Prop. 14]{green}. The calculations are naturally more complicated 
when dealing with groupoids, so we will prove two lemmas first, and then tackle the proof of the main theorem.

\begin{lem}
\label{lem:exchange1}
	Let $(\A, G, \alpha)$ be a separable groupoid dynamical system, and let $B$ be a separable $C^*$-algebra. Suppose 
	$(\pi, U)$ is a covariant representation of $(\A \omax B, G, \alpha \otimes \id)$, and write $\pi = \pi_A \omax \pi_B$, where
	$\pi_A$ and $\pi_B$ are representations of $A$ and $B$, respectively. Then $(\pi_A, U)$ is a covariant representation of $(\A, G, \alpha)$.
\end{lem}
\begin{proof}
	Let $\mu$ and $\go * \frah$ denote the quasi-invariant measure and Borel Hilbert bundle associated to $(\pi, U)$. Since $\pi_A \omax 
	\pi_B$ is a $C_0(\go)$-linear representation of $A \omax B$, we know from Proposition~\ref{prop:linrep} that $\pi_A$ 
	is $C_0(\go)$-linear, hence decomposable, and from Proposition~\ref{prop:linrep2} that $\pi_B$ is decomposable. Furthermore, 
	$(\pi_A \omax \pi_B)_u = \pi_{A, u} \omax \pi_{B, u}$ is a decomposition of $\pi_A \omax \pi_B$. Let $\nu = \mu \circ \lambda$ denote 
	the measure on $G$ induced from $\mu$ via the Haar system. The covariance of $(\pi, U)$ means that there is a $\nu$-null set 
	$N \subset G$ such that for all $\gamma \not\in N$,
	\begin{equation}
	\label{eq:covariant}
		U_\gamma \pi_{A, s(\gamma)} \bigl( a(s(\gamma)) \bigr) \pi_{B, s(\gamma)}(b) = \pi_{A, r(\gamma)} \bigl(\alpha_\gamma \bigl( 
		a(s(\gamma)) \bigr) \bigr) \pi_{B,r(\gamma)}(b) U_\gamma.
	\end{equation}
	for $a \in A$ and $b \in B$. Recall that there is a $\mu$-null set $E \subset \go$ such that $\pi_{B,u}$ is nondegenerate for all $u \not\in E$, 
	and we can take $E$ to be Borel. Put $V = r^{-1}(E) \cup s^{-1}(E)$. We claim that $\nu(V)=0$. Since $\chi_{r^{-1}(E)}$ is a positive Borel 
	function, \cite[Prop. 3.109]{geoff}, implies that $u \mapsto \lambda^u \bigl( r^{-1}(E) \bigr)$ is Borel, and
	\[
		\nu \bigl( r^{-1}(E) \bigr) = \int_{\go} \lambda^u \bigl( r^{-1}(E) \bigr) \, d\mu(u).
	\]
	Observe that $\lambda^u \bigl( r^{-1}(E) \bigr)$ can only be nonzero when $u \in r \bigl(r^{-1}(E) \bigr)=E$. But $\mu(E)=0$, so 
	\eqref{eq:func} is $\mu$-a.e. zero, and it follows that $\nu \bigl( r^{-1}(E) \bigr) = 0$. Since $\mu$ is quasi-invariant, $\nu^{-1}(r^{-1}(E)) = 0$ 
	as well. But then $\nu(s^{-1}(E)) = \nu^{-1}(r^{-1}(E)) = 0$, and it follows that $\nu(V)=0$.
	
	Define $\tilde{N} = N \cup V$. Then $\tilde{N}$ is $\nu$-null, the equation (\ref{eq:covariant}) holds for all $\gamma \not\in \tilde{N}$, 
	and the representations $\pi_{B, s(\gamma)}$ and $\pi_{B, r(\gamma)}$ are nondegenerate for all $\gamma \not \in \tilde{N}$. 
	Consequently,
	\[
		U_\gamma \pi_{A, s(\gamma)} \bigl( a(s(\gamma)) \bigr) = \pi_{A, r(\gamma)} \bigl( \alpha_\gamma(a(s(\gamma)) \bigr) U_\gamma
	\]
	for all $\gamma \not \in \tilde{N}$, and $(\pi_A, U)$ is a covariant representation.
\end{proof}

\begin{lem}
\label{lem:exchange2}
	Let $(\A, G, \alpha)$ be a separable groupoid dynamical system, $B$ a separable $C^*$-algebra, and $(\pi, U)$ a covariant 
	representation of $(\A \omax B, G, \alpha \otimes \id)$. If $\nu$ denotes the measure on $G$ induced by $\mu$, 
	then for $\nu$-almost all $\gamma$, 
	\begin{equation}
	\label{eq:covariant3}
		U_\gamma \pi_{B, s(\gamma)}(b) = \pi_{B, r(\gamma)}(b) U_\gamma
	\end{equation}
	for all $b \in B$. As a result, $\pi_B$ commutes with the integrated form $\pi_A \rtimes U$.
\end{lem}
\begin{proof}
	As we have already seen, the covariance of $(\pi, U)$ means that there is a $\nu$-null set $N \subset G$ such that
	\eqref{eq:covariant} holds for all $\gamma \not\in N$, $a \in A$, and $b \in B$. Furthermore, the proof of Lemma \ref{lem:exchange1} 
	guarantees that there is a $\nu$-null $\tilde{N} \subset G$ such that for all $\gamma \not\in \tilde{N}$, \eqref{eq:covariant} holds and the 
	representations $\pi_{A, s(\gamma)}$ and $\pi_{A, r(\gamma)}$ are nondegenerate. It follows that for all $\gamma \not\in \tilde{N}$,
	\[
		U_\gamma \pi_{B, s(\gamma)}(b) = \pi_{B, r(\gamma)}(b) U_\gamma
	\]
	for all $b \in B$, and the first claim is proven.
	
	For the second assertion, we need to explicitly compute with the integrated form of $(\pi_A, U)$. Let $f \in \Gamma_c(G, r^*\A)$ and $b 
	\in B$. Then for any $h, k \in L^2(\go*\frah, \mu)$,
	\begin{align*}
		&\ip{\pi_A\rtimes U(f)\pi_B(b)h}{k} = \\
		 &  \quad \quad = \int_G \ip{\pi_{A,r(\gamma)}(f(\gamma)) U_\gamma \pi_{B, s(\gamma)}(b) h(s(\gamma))}{k(r(\gamma))} \Delta
		 (\gamma)^{-1/2} \, d\nu(\gamma) \\
		&  \quad \quad = \int_G \ip{\pi_{A,r(\gamma)}(f(\gamma)) \pi_{B,r(\gamma)}(b) U_\gamma h(s(\gamma))}{k(r(\gamma))} \Delta
		(\gamma)^{-1/2} \, d\nu(\gamma)
	\end{align*}
	by (\ref{eq:covariant3}). Now continuing and using the fact that $\pi_A$ and $\pi_B$ commute, we have
	\begin{align*}
		&= \int_G \ip{\pi_{B,r(\gamma)}(b) \pi_{A, r(\gamma)}(f(\gamma)) U_\gamma h(s(\gamma))}{k(r(\gamma))} \Delta(\gamma)^{-1/2} \, d
		\nu(\gamma) \\
		&= \int_G \ip{\pi_{A,r(\gamma)}(f(\gamma)) U_\gamma h(s(\gamma))}{\pi_{B, r(\gamma)}(b)^*k(r(\gamma))} \Delta(\gamma)^{-1/2} \, d
		\nu(\gamma) \\
		&= \ip{\pi_A \rtimes U(f)h}{\pi_B(b)^*k} \\
		&= \ip{\pi_B(b) (\pi_A \rtimes U)(f) h}{k}.
	\end{align*}
	This holds for all $h, k \in L^2 \bigl( \go*\frah, \mu \bigr)$, so $\pi_B$ and $\pi_A \rtimes U$ commute.
\end{proof}

\begin{proof}[Proof of Theorem \ref{thm:exchange}]
	Let $\rho$ be a faithful representation of $(\A \omax B) \rtimes_{\alpha \otimes \id} G$. Then by Renault's Disintegration Theorem, $\rho$ 
	is equivalent to the integrated form of a covariant representation $(\pi, U)$ on a direct integral $L^2(\go*\frah, \mu)$, so we can assume 
	$\rho = \pi \rtimes U$. Moreover, we can write $\pi = \pi_A \omax \pi_B$, where $\pi_A$ and $\pi_B$ are representations of $A$ and $B$, 
	respectively, on $L^2 \bigl( \go*\frah, \mu \bigr)$ with commuting ranges. Hence $\rho = (\pi_A \omax \pi_B) \rtimes U$. By Lemma 
	\ref{lem:exchange1}, $(\pi_A, U)$ is a covariant representation of $(\A, G, \alpha)$, and $\pi_B$ commutes with $\pi_A \rtimes U$ by 
	Lemma \ref{lem:exchange2}. Thus we can form the representation $L = (\pi_A \rtimes U) \omax \pi_B$ of $(\A \rtimes_\alpha G) \omax B$, 
	which we claim has the same range as $\rho$.
	
	Since $\Gamma_c(G, r^*\A) \odot B$ is dense in $(\A \omax B) \rtimes_{\alpha \otimes \id} G$ by Proposition \ref{prop:tensorILT}, the 
	range of $\rho$ is generated by elements of the form $\rho(f \ohat b)$ for $f \in \Gamma_c(G, r^*\A)$ and $b \in B$. On such an elementary 
	tensor, we have 
	\begin{align*}
		\rho (f \ohat b) h(u) &= \int_G \pi_u \bigl( f(\gamma) \otimes b \bigr) U_\gamma h(s(\gamma)) \Delta(\gamma)^{-1/2} \, 
				d\lambda^u(\gamma) \\
			&= \int_G \pi_{A, u}(f(\gamma)) \pi_{B,u}(b) U_\gamma h(s(\gamma)) \modu \, d\lambda^u(\gamma) \\
			&= \pi_B(b) \int_G \pi_{A,u}(f(\gamma)) U_\gamma h(s(\gamma)) \modu \, d\lambda^u(\gamma) \\
			&= \pi_B(b) \pi_A \rtimes U(f) h(u) \\
			&= (\pi_A \rtimes U) \omax \pi_B (f \otimes b) h(u).
	\end{align*}
	for any $h \in L^2 ( \go*\frah, \mu)$. Thus
	\[
		(\pi_A \omax \pi_B) \rtimes U \bigl( f \ohat b \bigr) = (\pi_A \rtimes U) \omax \pi_B(f \otimes b) = L(f \otimes b)
	\]
	for all $f \in \Gamma_c(G, r^*\A)$ and $b \in B$. Since $(\A \rtimes_\alpha G) \omax B$ is generated by elementary tensors of this form, 
	it follows that $\rho$ and $L$ have the same range.
	
	Since $\rho$ is faithful and its range is the same as that of $L$, it makes sense to define a homomorphism $\Phi : (\A \rtimes_\alpha G) 
	\omax B \to (\A \omax B) \rtimes_{\alpha \otimes \id} G$ by $\Phi = \rho^{-1} \circ L$. Note that $\Phi$ is surjective and has the desired 
	form on elementary tensors simply by definition. To see that $\Phi$ is an isomorphism, we will construct an explicit inverse. Suppose 
	$M$ is a faithful representation of $(\A \rtimes_\alpha G) \omax B$ on a direct integral $L^2 \bigl( \go*\frack, \mu' \bigr)$. Then we can 
	decompose $M$ as $M = (\rho_A \rtimes V) \omax \rho_B$, where $\rho_B$ and $\rho_A \rtimes U$ have commuting ranges. Note 
	that on elementary tensors we have
	\[
		M \bigl( (z \ohat a) \otimes b \bigr)  = \rho_A(a) V(z) \rho_B(b) = \rho_B(b) \rho_A(a) V(z),
	\]
	so $\rho_A$ and $\rho_B$ have commuting ranges. Thus we can define a representation $\rho = (\rho_A \omax \rho_B) \rtimes V$ of 
	$(\A \omax B) \rtimes_{\alpha \otimes \id} G$. By the same arguments as before, $\rho$ and $M$ have common range, and we can define 
	$\Psi : (\A \omax B) \rtimes_{\alpha \otimes \id} G \to (\A \rtimes_\alpha G) \omax B$ by $\Psi = M^{-1} \circ \rho$. Then clearly 
	$\Psi \circ \Phi$ and $\Phi \circ \Psi$ act as the identity on elementary tensors, so $\Phi$ is an isomorphism with $\Psi = \Phi^{-1}$. 
	Furthermore, it is easy to see that any other such isomorphism would need to agree with $\Phi$ on elementary tensors, hence everywhere. 
	Thus $\Phi$ is natural and unique.
\end{proof}

In Section 6 we will prove an analogue of Theorem \ref{thm:exchange} for reduced crossed products and spatial tensor products. It will follow in 
part from Theorem \ref{thm:exchange}, along with some tools that we will obtain in the next section. 

\section{Nuclearity}
The goal of this section is to prove the first of our two main theorems, which gives sufficient conditions for a groupoid crossed product to be 
nuclear.

\begin{thm}
\label{thm:nucleargroupoid}
	Let $(\A, G, \alpha)$ be a separable groupoid dynamical system, and suppose $A$ is nuclear and $G$ is measurewise amenable. 
	Then $\A \rtimes_\alpha G$ is nuclear.
\end{thm}

There are a few issues that we need to clear up. Suppose $(\A, G, \alpha)$ is a separable groupoid dynamical system. To prove that 
$\A \rtimes_\alpha G$ is nuclear, we need to show that $(\A \rtimes_\alpha G) \omax B = (\A \rtimes_\alpha G) \otimes_\sigma B$ for any 
$C^*$-algebra $B$. This could be problematic if $B$ is not separable. In particular, we need to consider the tensor product dynamical system 
$(\A \omax B, G, \alpha \otimes \id)$, which we can only construct under the assumption that $B$ is separable. It would not be feasible to work 
with a nonseparable $C^*$-algebra $B$, so we need to somehow restrict to the separable case. The following clever trick is due to Dana 
Williams and Roger Smith.

\begin{prop}
\label{prop:sepnuclear}
	Let $A$ be a  $C^*$-algebra. Suppose that for any \emph{separable} $C^*$-algebra $B$, $A \odot B$ has a unique $C^*$-norm. Then 
	$A$ is nuclear.
\end{prop}
\begin{proof}
	Suppose that $C$ is a non-separable $C^*$-algebra for which $A \omax C \neq A \otimes_\sigma C$. Then there is some element 
	$\sum_{i=1}^n a_i \otimes c_i \in A \odot C$ such that
	\[
		\Bignorm{\sum a_i \otimes c_i}_\sigma < \Bignorm{\sum a_i \otimes c_i}_\text{max}.
	\]
	Put $B = C^*(\{c_1, \ldots, c_n\})$. Then $B$ is a separable $C^*$-subalgebra of $C$. Let $\norm{\cdot}_\gamma$ denote the 
	restriction of $\norm{\cdot}_\text{max}$ from $A \odot C$ 
	to $A \odot B$. Since the restriction of $\norm{\cdot}_\sigma$ to $A \odot B$ agrees with the spatial norm on $A \odot B$, we have
	\[
		\Bignorm{\sum a_i \otimes c_i}_\sigma < \Bignorm{\sum a_i \otimes c_i}_\gamma,
	\]
	and $\norm{\cdot}_\gamma$ is distinct from the spatial norm on $A \odot B$. In other words, $B$ is a separable $C^*$-algebra for which 
	$A \odot B$ has multiple $C^*$-norms, and we are done.
\end{proof}

We can now safely assume that $B$ is separable. Under this assumption, there are faithful representations $\rho_A$ and $\rho_B$ of $A$ 
and $B$, respectively, on \emph{separable} Hilbert spaces $\Hil_A$ and $\Hil_B$. Moreover, \cite[Thm. 8.3.2]{dixmier} allows us to assume 
that there is a Borel Hilbert bundle $\go * \frah$ and a finite Borel measure $\mu$ on $\go$ such that $\rho_A$ is a $C_0(\go)$-linear 
representation of $A$ on $\Hil_A = L^2(\go * \frah, \mu)$. We then define
\[
	\pi_A = \rho_A \otimes 1_{\Hil_B} \, \text{ and } \, \pi_B = 1_{\Hil_A} \otimes \rho_B.
\]
Then $\pi_A$ and $\pi_B$ are faithful representations of $A$ and $B$ on $L^2(\go*\frah, \mu) \otimes \Hil_B$ with commuting ranges. 
Thus there is an induced representation $\pi = \pi_A \omax \pi_B$ of $A \omax B$ on $L^2(\go*\frah, \mu) \otimes \Hil_B$. Moreover, since 
we are assuming $A$ is nuclear, $\pi$ is actually faithful.

In the following results, we will need to use the fact that $\pi$ can be viewed as a $C_0(\go)$-linear representation on a certain direct integral. 
The key to doing so is \cite[Prop. II.1.11]{dixVN}, which guarantees that there is an analytic Borel Hilbert bundle $\go * (\frah \otimes \Hil_B)$ 
whose fibers are simply $\Hil(x) \otimes \Hil_B$, and that $L^2(\go*\frah, \mu) \otimes \Hil_B$ is naturally isomorphic to 
$L^2(\go*(\frah \otimes \Hil_B), \mu)$. Therefore, we will frequently identify $\Hil_A \otimes \Hil_B$ with $\Hil = 
L^2(\go * (\frah \otimes \Hil_B), \mu)$. Moreover, Proposition \ref{prop:linrep} tells us that $\pi$ is a $C_0(\go)$-linear representation on 
$\Hil$ under this identification.

It is critical to verify that the induction process plays nicely with tensor products. The first step is to show that $\Ind \pi_A$ and $( \Ind \rho_A ) 
\otimes 1_{\Hil_B}$ are unitarily equivalent. Since $\rho_A : A \to B(\Hil_A)$, $\Ind \rho_A$ acts on $\Zz \otimes_A \Hil_A$, where $\Zz$ is the 
completion of $\Gamma_c(G, s^*\A)$ described at the end of Section 2. Thus $\Ind \rho_A \otimes 1_{\Hil_B}$ acts on $(\Zz \otimes_A \Hil_A) 
\otimes \Hil_B$. On the other hand, $\pi_A : A \to B(\Hil_A \otimes \Hil_B)$, so $\Ind \pi_A$ acts on $\Zz \otimes_A (\Hil_A \otimes \Hil_B)$.

\begin{prop}
\label{prop:inducedunitary}
	There is an isometry $U$ of $(\Zz_0 \odot \Hil_A) \odot \Hil_B$ onto $\Zz_0 \odot (\Hil_A \odot \Hil_B)$ 
	characterized on elementary tensors by
	\begin{equation}
	\label{eq:Udef}
		(z \otimes h) \otimes k \mapsto z \otimes (h \otimes k), \; z \in \Zz_0, h \in \Hil_A, k \in \Hil_B,
	\end{equation}
	which then extends to a unitary $U : (\Zz \otimes_A \Hil_A) \otimes \Hil_B \to \Zz \otimes_A (\Hil_A \otimes \Hil_B)$. Furthermore, $U$
	intertwines the representations $\Ind \pi_A$ and $( \Ind \rho_A ) \otimes 1_{\Hil_B}$.
\end{prop}
\begin{proof}
	Define $U$ on elementary tensors as in (\ref{eq:Udef}) and extend by linearity to all of $(\Zz_0 \odot \Hil_A) \odot \Hil_B$. Suppose
	$z_1, z_2 \in \Zz$, $h_1, h_2 \in \Hil_A$, and $k_1, k_2 \in \Hil_B$. Then we have
	\begin{align*}
		\ip{z_1 \otimes (h_1 \otimes k_1)}{z_2 \otimes (h_2 \otimes k_2)} &= \ip{\pi_A \bigl(\hip{z_2}{z_1}_A\bigr)(h_1 \otimes k_1)}{h_2 
		\otimes k_2} \\
			&= \bip{ (\rho_A \otimes 1_{\Hil_B}) \bigl(\hip{z_2}{z_1}_A\bigr)(h_1 \otimes k_1)}{h_2 \otimes k_2} \\
			&= \bip{\rho_A \bigl(\hip{z_2}{z_1}_A\bigr)h_1}{h_2} \ip{k_1}{k_2}.
	\end{align*}
	The first factor is precisely the inner product of $z_1 \otimes h_1$ and $z_2 \otimes h_2$ in $\Zz_0 \odot \Hil_A$, so
	\begin{align*}
		\ip{z_1 \otimes (h_1 \otimes k_1)}{z_2 \otimes (h_2 \otimes k_2)} &= \ip{z_1 \otimes h_1}{z_2 \otimes h_2} \ip{k_1}{k_2} \\
			&= \ip{(z_1 \otimes h_1) \otimes k_1}{(z_2 \otimes h_2) \otimes k_2}.
	\end{align*}
	Thus $U : (\Zz_0 \odot \Hil_A) \odot \Hil_B \to \Y_0 \odot (\Hil_A \odot \Hil_B)$ is isometric. It is clearly surjective, so it 
	extends to a unitary (also denoted $U$) from $(\Zz \otimes_A \Hil_A) \otimes \Hil_B$ onto $\Zz \otimes_A (\Hil_A 
	\otimes \Hil_B)$.

	Now let $f \in \Gamma_c(G, r^*\A)$, $z \in \Zz_0$, $h \in \Hil_A$, and $k \in \Hil_B$. Then
	\begin{align*}
		\Ind \pi_A(f) \cdot U \bigl( (z \otimes h) \otimes k \bigr) &= \Ind \pi_A(f) \bigl( z \otimes (h \otimes k) \bigr) \\
			&= f \cdot z \otimes (h \otimes k) \\
			&= U\bigl( (f \cdot z \otimes h) \otimes k \bigr) \\
			&= U \bigl( \Ind \rho_A(f)(z \otimes h) \otimes k \bigr) \\
			&= U \cdot \bigl( \Ind \rho_A \otimes 1_{\Hil_B} \bigr)(f) \bigl( (z \otimes h) \otimes k \bigr),
	\end{align*}
	so $U$ intertwines $\Ind \pi_A$ and $\Ind \rho_A \otimes 1_{\Hil_B}$.
\end{proof}

\begin{rem}
\label{rem:inducedB}
	In the proof of Theorem \ref{thm:nucleargroupoid}, we will also need a representation of $(\A \rtimes_\alpha G) \omax B$ on $\Zz \otimes_A 
	(\Hil_A \otimes \Hil_B)$. Of course we can do this by producing a representation $\tilde{\pi}_B$ of $B$ on 
	$\Zz \otimes_A (\Hil_A \otimes \Hil_B)$ whose range commutes with that of $\Ind \pi_A$. Form the representation $1_{\Zz \otimes \Hil} 
	\otimes \rho_B$ of $B$ on $(\Zz \otimes_A \Hil_A) \otimes \Hil_B$ and define
	\[
		\tilde{\pi}_B(b) = U (1_{\Zz \otimes \Hil_A} \otimes \rho_B(b)) U^*
	\]
	for all $b \in B$, where $U$ is the unitary from Proposition \ref{prop:inducedunitary}. Note then that
	\begin{align*}
		\tilde{\pi}_B(b) (z \otimes (h \otimes k)) &= U(1_{\Zz \otimes \Hil_A} \otimes \rho_B(b))U^*(z \otimes (h \otimes k)) \\
			&= U((z \otimes h) \otimes \rho_B(b)k) \\
			&= z \otimes (h \otimes \rho_B(b)k) \\
			&= z \otimes \pi_B(b)(h \otimes k).
	\end{align*}
	It is also straightforward to verify that $\Ind \pi_A$ and $\tilde{\pi}_B$ have commuting ranges: if $a \in \A \rtimes_\alpha G$ and 
	$b \in B$, then
	\begin{align*}
		\Ind \pi_A(a) \tilde{\pi}_B(b) &= U(\Ind \rho_A(a) \otimes 1_{\Hil_B})(1_{\Zz \otimes \Hil_A} \otimes \rho_B(b)) U^* \\
			&= U(1_{\Zz \otimes \Hil_A} \otimes \rho_B(b))(\Ind \rho_A(a) \otimes 1_{\Hil_B}) U^* \\
			&= \tilde{\pi}_B(b) \Ind \rho_A(a).
	\end{align*}
	Therefore, there is a representation $\Ind \pi_A \omax \tilde{\pi}_B$ of $(\A \rtimes_\alpha G) \omax B$ on $\Zz \otimes_A \Hil$.
\end{rem}

We have thus far considered only induced representations of $\A \rtimes_\alpha G$, but we will also need to look at representations of
$(\A \omax B) \rtimes_{\alpha \otimes \id} G$. In particular, we need to work with $\Ind (\pi_A \omax \pi_B)$. This representation acts on
the Hilbert space $\Y \otimes_{A \otimes B} \Hil$, where $\Y$ is the completion of the pre-imprimitivity bimodule $\Y_0 = \Gamma_c(G, 
s^*(\A \omax B))$ described in Section 2. Our final claim before we prove the main theorem is the following.

\begin{lem}
\label{lem:unitaryV}
	There is an isometry $V : \Y_0 \odot \Hil \to \Zz_0 \odot \Hil$, which is characterized on elementary 
	tensors by
	\[
		(f \ohat b) \otimes h \mapsto f \otimes \pi_B(b) h 
	\]
	for $f \in \Gamma_c(G, s^*\A)$, $b \in B$, and $h \in \Hil$. This isometry then extends to a unitary $V : \Y \otimes_{A \otimes B} \Hil 
	\to \Zz \otimes_A \Hil$.
\end{lem}
\begin{proof}
	We know from Proposition \ref{prop:tensorILT} that elementary tensors of the form $f \ohat b$ span a dense subspace of $\Y_0 =
	\Gamma_c(G, s^*(\A \omax B))$ with respect to the inductive limit topology. We claim that this implies density with respect to the norm 
	topology on $\Y$. To see this, suppose $\{y_i\}$ is a net in $\Y_0$ that converges to $y \in \Y$ with respect to the inductive limit topology. 
	Then
	\begin{align*}
		\norm{y_i - y}_A^2 &= \norm{\hip{y_i - y}{y_i - y}_A} \\
			&= \sup_{u \in \go} \norm{\hip{y_i - y}{y_i - y}_A(u)} \\
			&\leq \sup_{u \in \go} \int_G \norm{(y_i-y)(\xi)^*(y_i-y)(\xi)} \, d\lambda_u(\xi) \\
			&= \sup_{u \in \go} \int_G \norm{y_i(\xi)-y(\xi)}^2 \, d\lambda_u(\xi).
	\end{align*}
	Let $\varepsilon > 0$. Since $y_i \to y$ in the inductive limit topology, there is a compact set $K$ such that $\supp(y_i), \supp(y) \subset K$ 
	and $\norm{y_i(\xi) - y(\xi)} < \varepsilon^{1/2}$ for sufficiently large $i$. Therefore, for sufficiently large $i$, we have
	\[
		\norm{y_i - y}_A^2 < \varepsilon \cdot \sup_{u \in \go} \lambda_u(K).
	\]
	A slight modification of \cite[Lem. 1.23]{geoff} shows that $\sup_{u \in \go} \lambda_u(K)$ is finite, so it follows that $y_i \to y$ in norm.
	
	Thanks to this claim, it will suffice to compute with elementary tensors. Let $f_1, f_2 \in \Y_0$, $b_1, b_2 \in B$, and $h_1, h_2 \in \Hil$. Then
	\begin{align*}
		\ip{V((f_1 \ohat b_1) \otimes h_1)}{V((f_2 \ohat b_2) \otimes h_2)} &=\ip{(f_1 \otimes \pi_B(b_1)h_1}{(f_2 \otimes \pi_B(b_2)h_2} \\
		& = \ip{\pi_A \bigl( \hip{f_2}{f_1}_A \bigr) \pi_B(b_1)h_1}{\pi_B(b_2)h_2} \\
		& = \ip{\pi_A \bigl( \hip{f_2}{f_1}_A \bigr) \pi_B(b_2^*b_1)h_1}{h_2} \\
		& = \ip{\pi \bigl( \hip{f_2}{f_1}_A \otimes b_2^*b_1 \bigr) h_1}{h_2},
	\end{align*}
	where we have used the fact that $\pi_A$ and $\pi_B$ commute. Now for each $u \in \go$,
	\begin{align*}
		\pi \bigl( \hip{f_2}{f_1}_A \otimes b_2^*b_1 \bigr)(u) &= \pi_u \bigl( \hip{f_2}{f_1}_A(u) \otimes b_2^*b_1 \bigr) \\
			&= \pi_{A,u} \bigl( \hip{f_2}{f_1}_A(u) \bigr) \pi_{B,u}(b_2^*b_1),
	\end{align*}
	where
	\begin{align*}
		\pi_{A,u} \bigl( \hip{f_2}{f_1}_A(u) \bigr) &= \pi_{A,u} \biggl( \int_G f_2(\xi)^* f_1(\xi) \, d\lambda_u(\xi) \biggr) \\
			&= \int_G \pi_{A, u} \bigl( f_2(\xi)^* f_1(\xi) \bigr) \, d\lambda_u(\xi).
	\end{align*}
	Now
	\begin{align*}
		\pi \bigl( \hip{f_2}{f_1}_A \otimes b_2^*b_1 \bigr)(u) &= \int_G \pi_{A, u} \bigl( f_2(\xi)^* f_1(\xi) \bigr) \pi_{B,u} (b_2^*b_1) \, 
			d\lambda_u(\xi) \\
			&= \int_G \pi_u \bigl( (f_2(\xi) \otimes b_2)^*(f_1(\xi) \otimes b_1) \bigr) \, d\lambda_u(\xi) \\
			&= \pi_u \biggl( \int_G (f_2 \ohat b_2)(\xi)^* (f_1 \ohat b_1)(\xi) \, d\lambda_u(\xi) \biggr) \\
			&= \pi \bigl( \hip{f_2 \otimes b_2}{f_1 \otimes b_1}_{A \otimes B} \bigr)(u).
	\end{align*}
	Therefore,
	\begin{align*}
		\ip{V((f_1 \ohat b_1) \otimes h_1)}{V((f_2 \ohat b_2) \otimes h_2)} &=\ip{\pi \bigl( \hip{f_2 \ohat b_2}{f_1 \ohat b_1}_{A \otimes B} 
			\bigr)h_1}{h_2} \\
		&= \ip{(f_1 \ohat b_1) \otimes h_1}{(f_2 \ohat b_2) \otimes h_2},
	\end{align*}
	so $V$ is an isometry. Since $\pi_B$ is nondegenerate, $V$ is clearly surjective and thus extends to a unitary $V : \Y \otimes \Hil \to 
	\Zz \otimes \Hil$.
\end{proof}

\begin{proof}[Proof of \ref{thm:nucleargroupoid}]
	Let $\rho_A : A \to B(\Hil_A)$ and $\rho_B : B \to B(\Hil_B)$ be the faithful representations of $A$ and $B$ described above, 
	and let $\pi_A = \rho_A \otimes 1_{\Hil_B}$ and $\pi_B = 1_{\Hil_A} \otimes \rho_B$. Then $\pi_A$ and $\pi_B$ are faithful, and since 
	$A$ is nuclear, the representation $\pi = \pi_A \omax~\pi_B$ of $A \omax B$ on $\Hil = \Hil_A \otimes \Hil_B$ is faithful. Moreover, since 
	$G$ is measurewise amenable, the induced representation $\Ind \pi$ of $(\A \omax B) \rtimes_{\alpha \otimes \id} G$ on $\Y
	\otimes_{A \otimes B} \Hil$ is also faithful as a consequence of Theorem 1 of \cite{aidan-dana}.
	
	Recall from Theorem \ref{thm:exchange} that there is an isomorphism $\Phi : (\A \rtimes_\alpha G) \omax B \to (\A \omax B) 
	\rtimes_{\alpha \otimes \id} G$, so $\Ind \pi \circ \Phi$ is a faithful representation of $(\A \rtimes_\alpha G) \omax B$. We claim that 
	the unitary $V$ of Lemma \ref{lem:unitaryV} intertwines this representation with $\Ind \pi_A \omax \tilde{\pi}_B$. Let $f \in \Gamma_c(G, 
	r^*\A)$, $y \in \Y_0$, $b, c \in B$, and $h \in \Hil$. Then
	\begin{align*}
		\Ind \pi_A \omax \tilde{\pi}_B(f \otimes b) V ((y \ohat c) \otimes h) &= \Ind \pi_A(f) \tilde{\pi}_B(b) (y \otimes \pi_B(c)h) \\
			&= \Ind \pi_A(f) (y \otimes \pi_B(bc)h) \\
			&= f \cdot y \otimes \pi_B(bc)h \\
			&= V((f \cdot y \ohat bc) \otimes h).
	\end{align*}
	Recall that
	\[
		f \cdot y(\gamma) = \int_G \alpha_\gamma^{-1}(f(\eta)) y(\eta^{-1}\gamma) \, d\lambda^{r(\gamma)}(\eta),
	\]
	so
	\begin{align*}
		(f \cdot y \ohat bc)(\gamma) &= f \cdot y(\gamma) \otimes bc \\
			&= \int_G \alpha_\gamma^{-1}(f(\eta)) y(\eta^{-1}\gamma) \otimes bc \, d\lambda^{r(\gamma)}(\eta) \\
			&= \int_G \bigl( \alpha_\gamma^{-1}(f(\eta)) \otimes b \bigr) \bigl( y(\eta^{-1}\gamma) \otimes c \bigr) \, 
				d\lambda^{r(\gamma)}(\eta) \\
			&= \int_G (\alpha \otimes \id)_\gamma^{-1}(f \ohat b(\eta)) (y \ohat c(\eta^{-1}\gamma)) \, d\lambda^{r(\gamma)}(\eta) \\
			&= (f \ohat b) \cdot (y \ohat c)(\gamma).
	\end{align*}
	Therefore,
	\begin{align*}
		V((f \cdot y \ohat bc) \otimes h) &= V((f \ohat b) \cdot (y \ohat c) \otimes h) \\
			&= V \cdot \Ind \pi(f \ohat b)((y \ohat c) \otimes h) \\
			&= V \cdot \Ind \pi \circ \Phi(f \ohat b)((y \ohat c) \otimes h).
	\end{align*}
	Since the tensors of the form $(y \ohat c) \otimes h$ span a dense subspace of $\Y \otimes_{A \otimes B} \Hil$, $\Ind \pi_A \omax 
	\tilde{\pi}_B(f \otimes b) V = V \Ind \pi \circ \Phi(f \ohat b)$. Similarly, $\Gamma_c(G, r^*\A) \odot B$ is dense in $\Gamma_c(G, 
	r^*(\A \omax B))$ with respect to the inductive limit topology, so it follows that $V$ intertwines $\Ind \pi_A \omax \tilde{\pi}_B$ and 
	$\Ind \pi \circ \Phi$. Thus $\Ind \pi_A \omax \tilde{\pi}_B$ is faithful.
	
	Now let $\kappa : (\A \rtimes_{\alpha} G) \omax B \to (\A \rtimes_\alpha G) \otimes_\sigma B$ be the canonical quotient map. Then we 
	have
	\[
		\bigl( \Ind \rho_A \otimes 1_{\Hil_B} \bigr) \omax \left( 1_{\Zz \otimes \Hil_A} \otimes \rho_B \right) = \bigl( \Ind \rho_A \otimes 
		\rho_B \bigr) \circ \kappa.
	\]
	We have already seen that the unitary $U$ of Proposition \ref{prop:inducedunitary} intertwines $\Ind \rho_A \otimes 1_{\Hil_B}$ with 
	$\Ind \pi_A$ and
	$1_{\Zz \otimes \Hil_A} \otimes \rho_B$ with $\tilde{\pi}_B$, so the left side is equivalent to $\Ind \pi_A \omax \tilde{\pi}_B$.
	This representation is faithful, as is $\Ind \rho_A \otimes \rho_B$, which forces $\kappa$ to be injective. Thus 
	\[
		(\A \rtimes_\alpha G) \omax B = (\A \rtimes_\alpha G) \otimes_\sigma B
	\]
	for any separable $C^*$-algebra $B$. By Proposition \ref{prop:sepnuclear}, this is enough to guarantee that $\A \rtimes_\alpha G$ is nuclear.
\end{proof}

\section{Exactness}
With the nuclearity theorem out of the way, we now aim to prove a related theorem on exactness for groupoid 
crossed products. It is a generalization of a theorem of Kirchberg, who showed in \cite[Prop. 7.1(v)]{kirchberg} that the crossed product of 
an exact $C^*$-algebra by an amenable group is again exact. His proof hinges upon the fact that the \emph{reduced} crossed product of an
exact $C^*$-algebra by an exact group is exact. We will do the same here for groupoids.

One of the ingredients that we will need is the promised analogue of Theorem \ref{thm:exchange} for reduced crossed products. It will mostly
follow from the earlier result for full crossed products, but there are still some technical details to resolve.

\begin{thm}
\label{thm:reducedexchange}
	There is a natural isomorphism
	\[
		\Psi : (\A \osig B) \rtimes_{\alpha \otimes \id, r} G \to (\A \rtimes_{\alpha, r} G) \osig B
	\]
	which is characterized on elementary tensors by
	\[
		\Psi ( f \ohat b ) = f \otimes b
	\]
	for $f \in \Gamma_c(G, r^*\A)$ and $b \in B$.
\end{thm}

The plan is to show that $\Psi$ is induced from the isomorphism $\Phi$ of Theorem \ref{thm:exchange} via the natural 
quotient map $\kappa : A \omax B \to A \osig B$. We first need to know that $\kappa$ is a \emph{$G$-equivariant} homomorphism, so that it 
induces a map $\kappa \rtimes \id : (\A \omax B) \rtimes_{\alpha \otimes \id} G \to (\A \osig B) \rtimes_{\alpha \otimes \id} G$.  
There does not seem to be any mention of $G$-equivariant homomorphisms in the literature yet, so we will develop some facts.

Suppose that $(\A, G, \alpha)$ and $(\B, G, \beta)$ are groupoid dynamical systems. Na\"{i}vely, a $G$-equivariant homomorphism 
$\varphi : A \to B$ should commute with the actions of $G$ on $\A$ and $\B$. Since we really need to work with the fibers of $\A$ and $\B$, 
it is necessary that $\varphi$ induce fiberwise homomorphisms. The natural way to do this is to require that $\varphi$ be a $C_0(\go)$-linear 
homomorphism, in which case we can appeal to Remark \ref{rem:bundlecorrespondence}. Thus $\varphi$ induces homomorphisms 
$\varphi_u : \A_u \to \B_u$ for all $u \in \go$, and we can use these to define $G$-equivariance.

\begin{defn}
	Let $G$ be a locally compact Hausdorff groupoid, and let $(\A, G, \alpha)$ and $(\B, G, \beta)$ be dynamical systems. A 
	$C_0(\go)$-linear homomorphism $\varphi : A \to B$ is called \emph{$G$-equivariant} if
	\[
		\varphi_{r(\gamma)} \bigl( \alpha_\gamma(a) \bigr) = \beta_\gamma \bigl( \varphi_{s(\gamma)}(a) \bigr)
	\]
	for all $\gamma \in G$ and $a \in \A_{s(\gamma)}$.
\end{defn}

As in \cite[Cor. 2.48]{TFB2} for groups, we would like to show that any $G$-equivariant homomorphism induces a homomorphism between the 
associated crossed products in a natural way. Recall that Propositions \ref{prop:pullbackhom} and \ref{prop:pullbackhom2} guarantee that 
there is a homomorphism $r^*\varphi : r^*A \to r^*B$ characterized by
\[
	r^*\varphi(f)(\gamma) = \varphi_{r(\gamma)} \bigl( f(\gamma) \bigr)
\]
for all $f \in \Gamma_0(G, r^*\A) \cong r^*A$ and $\gamma \in G$. Furthermore, $r^*\varphi$ takes compactly 
supported sections to compactly supported sections. It shouldn't be surprising that the restriction of $r^*\varphi$ to 
$\Gamma_c(G, r^*\A)$ is a homomorphism that extends to $\A \rtimes_\alpha G$.

\begin{prop}
\label{prop:Gequivariant}
	Let $(\A, G, \alpha)$ and $(\B, G, \beta)$ be groupoid dynamical systems, and let $\varphi : A \to B$ be a $G$-equivariant 
	homomorphism. Then there is a homomorphism $\varphi \rtimes \id : \A \rtimes_\alpha G \to \B \rtimes_\beta G$, which takes 
	$\Gamma_c(G, r^*\A)$ into $\Gamma_c(G, r^*\B)$, and
	\begin{equation}
	\label{eq:equivariant}
		\varphi \rtimes \id(f)(\gamma) = \varphi_{r(\gamma)}(f(\gamma))
	\end{equation}
	for $f \in \Gamma_c(G, r^*\A)$.
\end{prop}
\begin{proof}
	Define $\varphi \rtimes \id$ on $\Gamma_c(G, r^*\A)$ as in (\ref{eq:equivariant}). We have already observed that $\varphi \rtimes \id$ is 
	linear and maps into $\Gamma_c(G, r^*\B)$, and we claim now that it is a homomorphism. If $f, g \in \Gamma_c(G, r^*\A)$, then we have
	\begin{align*}
		\bigl( \varphi \rtimes \id(f) \bigr) * \bigl( \varphi \rtimes \id(g) \bigr)(\gamma) &= \int_G \varphi \rtimes \id(f)(\eta) \beta_\eta \bigl( 
			\varphi \rtimes \id(g)(\eta^{-1}\gamma) \bigr) \, d\lambda^{r(\gamma)}(\eta) \\
			&= \int_G \varphi_{r(\eta)}(f(\eta)) \beta_\eta \bigl( \varphi_{r(\eta^{-1}\gamma)} \bigl( g(\eta^{-1}\gamma) \bigr) \bigr) \, 
				d\lambda^{r(\gamma)}(\eta) \\
			&= \int_G \varphi_{r(\eta)}(f(\eta)) \beta_\eta \bigl( \varphi_{s(\eta)} \bigl( g(\eta^{-1}\gamma) \bigr) \bigr) \, 
				d\lambda^{r(\gamma)}(\eta).
	\end{align*}
	Since $\varphi$ is $G$-equivariant, this becomes
	\begin{align*}
		&= \int_G \varphi_{r(\gamma)}(f(\eta)) \varphi_{r(\eta)} \bigl( \alpha_\eta \bigl(g(\eta^{-1}\gamma) \bigr) \bigr) \, 
			d\lambda^{r(\gamma)}(\eta) \\
		&= \int_G \varphi_{r(\gamma)} \bigl(f(\eta) \alpha_\eta \bigl( g(\eta^{-1}\gamma) \bigr) \bigr) \, d\lambda^{r(\gamma)}(\eta) \\
		&= \varphi_{r(\gamma)} \bigl( f*g(\gamma) \bigr) \\
		&= \varphi \rtimes \id(f*g)(\gamma),
	\end{align*}
	so $\varphi \rtimes \id$ is multiplicative. Similarly, if $f \in \Gamma_c(G, r^*\A)$, we have
	\begin{align*}
		\varphi \rtimes \id(f^*)(\gamma) &= \varphi_{r(\gamma)} \bigl( \alpha_\gamma \bigl( f(\gamma^{-1})^* \bigr) \bigr) \\
			&= \beta_\gamma \bigl( \varphi_{s(\gamma)} \bigl( f(\gamma^{-1})^* \bigr) \bigr) \\
			&= \beta_\gamma \bigl( \varphi \rtimes \id(f)(\gamma^{-1})^* \bigr) \\
			&= \varphi \rtimes \id(f)^*(\gamma).
	\end{align*}
	Therefore, $\varphi \rtimes \id$ is a $*$-homomorphism.
	
	The only thing left to check is that $\varphi \rtimes \id$ extends to $\A \rtimes_\alpha G$. We will do this by showing that it is bounded 
	with respect to the $I$-norm (Equation 4.2 of \cite{mw08}) on $\Gamma_c(G, r^*\A)$. If $f \in \Gamma_c(G, r^*\A)$, then
	\[
		\int_G \norm{\varphi \rtimes \id(f)(\gamma)} \, d\lambda^u(\gamma) = \int_G \norm{\varphi_{r(\gamma)}(f(\gamma))} \, 
			d\lambda^u(\gamma) \leq 
			\int_G \norm{f(\gamma)} \, d\lambda^u(\gamma)
	\]
	for all $u \in \go$. A similar computation shows that $\int_G \norm{\varphi \rtimes \id(f)(\gamma)} \, d\lambda_u(\gamma) \leq  
	\int_G \norm{f(\gamma)} \, d\lambda_u(\gamma)$, so it follows that $\norm{\varphi \rtimes \id(f)}_I \leq \norm{f}_I$. Thus 
	$\varphi \rtimes \id$ is $I$-norm decreasing, so it extends to a homomorphism $\varphi \rtimes \id : \A \rtimes_\alpha G \to 
	\B \rtimes_\beta G$.
\end{proof}

We'll now apply this machinery to a very particular homomorphism. Let $(\A, G, \alpha)$ be a separable groupoid dynamical system, and 
let $B$ be a separable $C^*$-algebra. Let $\kappa : A \omax B \to A \otimes_\sigma B$ be the canonical quotient map.

\begin{prop}
	The homomorphism $\kappa : A \omax B \to A \otimes_\sigma B$ is $G$-equivariant, and thus induces a homomorphism 
	$\kappa \rtimes \id : (\A \omax B) \rtimes_{\alpha \otimes \id} G \to (\A \otimes_\sigma B) \rtimes_{\alpha \otimes \id} G$.
\end{prop}
\begin{proof}
	It is easy to verify that $\kappa$ is $C_0(\go)$-linear, for on elementary tensors
	\[
		\kappa \bigl( f \cdot (a \omax b) \bigr) = \kappa ( f \cdot a \omax b) = f \cdot a \osig b = f \cdot \kappa(a \omax b)
	\]
	for any $f \in C_0(\go)$. (Here we write $a \omax b$ and $a \osig b$ to emphasize where each elementary tensor lives.) Thus 
	$\kappa$ admits a fiberwise decomposition $\kappa_u : \A_u \omax B \to \A_u \otimes_\sigma B$ by Remark 
	\ref{rem:bundlecorrespondence}. We claim that for each $u \in \go$, $\kappa_u$ is nothing more than the canonical quotient map 
	$\A_u \omax B \to \A_u \osig B$. If $a \in A$ and $b \in B$, then for any $u \in \go$ we have by definition
	\[
		\kappa_u(a(u) \omax b) = \kappa_u\bigl( (a \omax b)(u) \bigr) = \kappa(a \omax b)(u).
	\]
	The right hand side is $(a \osig b)(u) = a(u) \osig b$, so $\kappa_u$ is the desired homomorphism.
	
	It remains to verify that $\kappa$ is indeed $G$-equivariant. Let $\gamma \in G$, $a \in \A_{s(\gamma)}$, and $b \in B$. Then we have
	\begin{align*}
		\kappa_{r(\gamma)} \bigl( \alpha_\gamma \omax \id(a \omax b) \bigr) &= \kappa_{r(\gamma)} \bigl( \alpha_\gamma(a) \omax b \bigr) \\
			&= \alpha_\gamma(a) \osig b \\
			&= \alpha_\gamma \osig \id(a \osig b) \\
			&= \alpha_\gamma \osig \id \bigl( \kappa_{s(\gamma)} (a \omax b) \bigr).
	\end{align*}
	Thus $\kappa$ respects the $G$-actions on $\A$ and $\B$. It then follows from Proposition \ref{prop:Gequivariant} that $\kappa$ 
	induces a homomorphism $\kappa \rtimes \id : (\A \omax B) \rtimes_{\alpha \otimes \id} G \to (\A \otimes_\sigma B) \rtimes_{\alpha \otimes 
	\id} G$.
\end{proof}
 
Now let $\rho_A$ and $\rho_B$ be faithful separable representations of $A$ and $B$ on Hilbert spaces $\Hil_A$ and $\Hil_B$, respectively, 
and put $\pi_A = \rho_A \otimes 1_{\Hil_B}$ and $\pi_B = 1_{\Hil_A} \otimes \rho_B$. We can assume that there is an analytic Borel Hilbert 
bundle $\go*\frah$ and a finite Borel measure $\mu$ on $\go$ such that $\rho_A$ is a $C_0(\go)$-linear representation of $A$ on 
$L^2(\go*\frah, \mu)$. Then $\rho_A$ is decomposable, and we have
\[
	\pi_A = \rho_A \otimes 1_{\Hil_B} = \int^\oplus_{\go} \rho_{A, u} \otimes 1_{\Hil_B} \, d\mu(u),
\]
so $\pi_A$ is decomposable as well, with $\pi_{A, u} = \rho_{A, u} \otimes 1_{\Hil_B}$. With this setup in place, the first result we need
to establish is that
\[
	\Ind(\rho_A \otimes \rho_B) \circ (\kappa \rtimes \id) = \Ind(\pi_A \omax \pi_B).
\]
We already know that $\Ind(\pi_A \omax \pi_B)$ acts on the Hilbert space $\Y \otimes_{A \omax B} \Hil$, where $\Y$ is a completion of the
Hilbert module $\Y_0 = \Gamma_c(G, s^*(\A \omax B))$. On the other hand, $\Ind (\rho_A \otimes \rho_B)$ acts on 
$\W \otimes_{A \osig B} \Hil$, where $\W$ arises as a completion of $\W_0 = \Gamma_c(G, s^*(\A \osig B))$. We need 
to reconcile this somehow.

\begin{lem}
	There is an isometry $T : \Y_0 \odot \Hil \to \W_0 \odot \Hil$, characterized on elementary tensors by
	\[
		T(y \otimes h) = s^*\kappa(y) \otimes h
	\]
	for $y \in \Y_0$ and $h \in \Hil$. This isometry extends to a unitary $T : \Y \otimes_{A \omax B} \Hil \to \W \otimes_{A \osig B} \Hil$, which
	intertwines $\Ind(\rho_A \otimes \rho_B) \circ (\kappa \rtimes \id)$ and $\Ind(\pi_A \omax \pi_B)$.
\end{lem}
\begin{proof}
	Define $T$ on elementary tensors as above. Then for $y_1, y_2 \in \Y_0$ and $h_1, h_2 \in \Hil$, we have
	\begin{align*}
		&\ip{T(y_1 \otimes h_1)}{T(y_2 \otimes h_2)} = \ip{s^*\kappa(y_1) \otimes h_1}{s^*\kappa(y_2) \otimes h_2} \\
			&\quad \quad \quad = \ip{\rho_A \otimes \rho_B \bigl( \hip{s^*\kappa(y_2)}{s^*\kappa(y_1)}_{A \osig B} \bigr)h_1}{h_2} \\
			&\quad \quad \quad = \int_G \ip{(\rho_A \otimes \rho_B)_u \bigl( \hip{s^*\kappa(y_2)}{s^*\kappa(y_1)}_{A \osig B}(u) \bigr) 
				h_1(u)}{h_2(u)} \, d\mu(u).
	\end{align*}
	Now
	\begin{align*}
		\hip{s^*\kappa(y_2)}{s^*\kappa(y_1)}_{A \osig B}(u) 
			&= \int_G \kappa_{s(\xi)}(y_2(\xi))^* \kappa_{s(\xi)}(y_1(\xi)) \, d\lambda_u(\xi) \\
			&= \kappa_u \biggl( \int_G y_2(\xi)^* y_1(\xi) \, d\lambda_u(\xi) \biggr) \\
			&= \kappa_u \bigl( \hip{y_2}{y_1}_{A \omax B}(u) \bigr),
	\end{align*}
	so
	\begin{align*}
		\ip{T(y_1 \otimes h_1)}{T(y_2 \otimes h_2)} &= \ip{\rho_A \otimes \rho_B \circ \kappa \bigl( \hip{y_2}{y_1}_{A \omax B} \bigr) 
			h_1}{h_2} \\
			&= \ip{\pi_A \omax \pi_B \bigl( \hip{y_2}{y_1}_{A \omax B} \bigr) h_1}{h_2} \\
			&= \ip{y_1 \otimes h_1}{y_2 \otimes h_2}.
	\end{align*}
	Therefore $T$ is isometric. It is clearly surjective, so it then extends to a unitary $T : \Y \otimes_{A \omax B} \Hil \to \W \otimes_{A 
	\osig B} \Hil$. 
	
	All that is left is the verification that $T$ intertwines $\Ind(\rho_A \otimes \rho_B) \circ (\kappa \rtimes \id)$ and $\Ind(\pi_A \omax \pi_B)$.
	Let $f \in \Gamma_c(G, s^*\A)$, $w \in \W_0$, $b, c \in B$, and $h \in \Hil$. Then
	\[
		\Ind(\rho_A \otimes \rho_B) \circ (\kappa \rtimes \id)(f \ohat b) T((w \ohat c) \otimes h) = \kappa \rtimes \id(f \ohat b) \cdot 
		s^*\kappa(w \ohat c) \otimes h,
	\]
	where
	\begin{align*}
		&\kappa \rtimes \id(f \ohat b) \cdot s^*\kappa(w \ohat c)(\gamma) \\
		&\quad \quad \quad = \int_G (\alpha \otimes \id)_\gamma^{-1}(\kappa \rtimes \id(f \ohat b)(\eta)) w \ohat c(\eta^{-1}\gamma) \, 
			d\lambda^{r(\gamma)}(\eta) \\
		& \quad \quad \quad = \int_G(\alpha \otimes \id)_\gamma^{-1}( \kappa_{r(\gamma)}(f(\gamma) \omax b) )(w(\eta^{-1}\gamma) \otimes 
			c) \, d\lambda^{r(\gamma)}(\eta) \\
		& \quad \quad \quad = \biggl( \int_G \alpha_\gamma^{-1}(f(\eta)) w(\eta^{-1}\gamma) \, d\lambda^{r(\gamma)}(\eta) \biggr) \otimes bc\\
		& \quad \quad \quad = f \cdot w(\gamma) \otimes bc \\
		& \quad \quad \quad = \kappa_{s(\gamma)} \bigl( f \cdot w(\gamma) \omax bc \bigr).
	\end{align*}
	Thus
	\begin{align*}
		\kappa \rtimes \id(f \ohat b) \cdot s^*\kappa(w \ohat c) \otimes h &= s^*\kappa \bigl( f \cdot w \ohat bc \bigr) \otimes h \\
			&= T \bigl( (f \cdot w \ohat bc) \otimes h \bigr).
	\end{align*}
	We have seen in previous calculations that
	\[
		\Ind(\pi_A \omax \pi_B)(f \ohat b)((w \ohat c) \otimes h) = (f \cdot w \ohat bc) \otimes h,
	\]
	so we have shown
	\[
		\Ind(\rho_A \otimes \rho_B) \circ (\kappa \rtimes \id)(f \ohat b) \cdot T = T \cdot \Ind(\pi_A \omax \pi_B)(f \ohat b). \qedhere
	\]
\end{proof}

The point of the previous lemma is the following: since the image of $\Ind(\rho_A \otimes \rho_B)$ is a concrete realization of 
$(\A \osig B) \rtimes_{\alpha \otimes \id, r} G$, we have a natural identification of $(\A \osig B) \rtimes_{\alpha \otimes \id, r} G$ with the 
image of $\Ind(\pi_A \omax \pi_B)$. A similar (though easier) result is given below.

\begin{lem}
	With all notation as above, we have a natural identification of the image of $\bigl( \Ind \pi_A \bigr) \omax \tilde{\pi}_B$ with 
	$(\A \rtimes_{\alpha, r} G) \osig B$.
\end{lem}
\begin{proof}
	The image of $\Ind \pi_A$ can be identified with $\A \rtimes_{\alpha, r} G$, and we know that $\Ind \pi_A$ and $(\Ind \rho_A) 
	\otimes 1_{\Hil_B}$ are equivalent by Proposition \ref{prop:inducedunitary}. We also know from Remark \ref{rem:inducedB} that
	$\tilde{\pi}_B = 1_{\Zz \otimes \Hil_A} \otimes \rho_B$, so if $\kappa' : (\A \rtimes_{\alpha} G) \omax B \to (\A \rtimes_{\alpha} G) \osig B$ 
	denotes the canonical map, we have
	\[
		\bigl( \Ind \pi_A \bigr) \omax \tilde{\pi}_B = \bigl( \Ind \rho_A \bigr) \otimes \rho_B  \circ \kappa'.
	\]
	Since $\rho_B$ is a faithful representation of $B$, we see that the image of $( \Ind \rho_A) \otimes \rho_B$ is naturally identified with 
	$(\A \rtimes_{\alpha, r} G) \osig B$. The result then follows.
\end{proof}

\begin{proof}[Proof of Theorem \ref{thm:reducedexchange}]
	Again by Theorem \ref{thm:exchange} and Lemma \ref{lem:unitaryV} there is a natural isomorphism $\Phi : (\A \omax B) \rtimes_{\alpha 
	\otimes \id} G \to (\A \rtimes_\alpha G) \omax B$ and a unitary $V : \Y \otimes_{A \otimes B} \Hil \to \Zz \otimes_A \Hil$ that together 
	intertwine $\Ind(\pi_A \omax \pi_B)$ and $\bigl( \Ind \pi_A \bigr) \omax \tilde{\pi}_B$. Therefore, the previous two lemmas imply that we 
	can identify $(\A \osig B) \rtimes_{\alpha \otimes \id, r} G$ and $(\A \rtimes_{\alpha, r} G) \osig B$ (which are the images of 
	$\Ind(\pi_A \omax \pi_B)$ and $(\Ind \rho_A) \otimes \rho_B$, respectively). Therefore, we obtain a commutative diagram
	\[
		\xymatrix{ (\A \omax B) \rtimes_{\alpha \otimes \id} G \ar[r]^\Phi \ar[d]_{\Ind(\pi_A \omax \pi_B)} & (\A \rtimes_\alpha G) \omax B 
			\ar[d]^{( \Ind \pi_A ) \omax \tilde{\pi}_B} \\
		(\A \osig B) \rtimes_{\alpha \otimes \id, r} G \ar[r]^\Psi & (\A \rtimes_{\alpha, r} G) \osig B
		}
	\]
	with $\Psi$ denoting the aforementioned identification. It is then clear that $\Psi$ has the desired properties.
\end{proof}

Recall that a locally compact group $G$ is \emph{exact} if whenever $(A, G, \alpha)$ is a dynamical system and $I$ is a 
$G$-invariant ideal in $A$, the sequence
\[
	\xymatrix{
		0 \ar[r] & I \rtimes_{\alpha \vert_I, r} G \ar[r] & A \rtimes_{\alpha, r} G \ar[r] & A/I \rtimes_{\alpha^I, r} G \ar[r] & 0 
	}
\]
of reduced crossed products is exact. We'd like to generalize this notion and study \emph{exact groupoids}.

Throughout this discussion, let $(\A, G, \alpha)$ be a separable groupoid dynamical system. Suppose $I$ is an ideal in $A$. We need to 
develop a criterion for when $I$ is invariant under the action of $G$ on $A$, which will ensure that we can build 
dynamical systems $(\I, G, \alpha \vert_I)$ and $(\A/\I, G, \alpha^I)$. This requires $I$ and $A/I$ 
to be equipped with $C_0(\go)$-algebra structures, which was shown in Propositions \ref{prop:C0Xideal} and 
\ref{prop:C0Xquotient}.

\begin{defn}
	The ideal $I$ is said to be \emph{$G$-invariant} if for all $\gamma \in G$, we have
	\[
		\alpha_\gamma \bigl( \I_{s(\gamma)} \bigr) = \I_{r(\gamma)}.
	\]
\end{defn}

\begin{rem}
	Note that the above definition implies that the restriction $\alpha \vert_I = \{ \alpha_\gamma \vert_{\I_{s(\gamma)}} \}_{\gamma \in G}$ 
	yields an action of $G$ on $\I$, and that the inclusion map $\iota : I \to A$ is $G$-equivariant. Furthermore, for each $\gamma \in G$ 
	we get an isomorphism
	\[
		\alpha_\gamma^I : (\A/\I)_{s(\gamma)} \to (\A/\I)_{r(\gamma)}.
	\]
	Under the natural identification of $(\A/\I)_u$ with $\A_u/\I_u$, this action is just 
	\[
		\alpha_\gamma^I \bigl( a(s(\gamma)) + \I_{s(\gamma)} \bigr) = \alpha_\gamma \bigl( a(s(\gamma)) \bigr) + \I_{r(\gamma)}.
	\]
	In other words, the quotient map $q : A \to A/I$ is $G$-equivariant.
\end{rem}

Since the maps $\iota$ and $q$ are $G$-equivariant, Proposition \ref{prop:Gequivariant} guarantees that they yield maps $\iota \rtimes \id : 
\I \rtimes_{\alpha \vert_I} G \to \A \rtimes_\alpha G$ and $q \rtimes \id : \A \rtimes_\alpha G \to (\A/\I) \rtimes_{\alpha^I} G$. Furthermore, it is
shown in \cite[Lem. 6.3.2]{ananth-renault} that the sequence
\begin{equation}
\label{eq:invariantsequence}
	\xymatrix{
		0 \ar[r] & \I \rtimes_{\alpha \vert_I} G \ar[r]^{\iota \rtimes \id} & \A \rtimes_\alpha G \ar[r]^(0.45){q \rtimes \id} & \A/\I \rtimes_{\alpha^I} 
		G \ar[r] & 0.
	}
\end{equation}
is exact. This fact also follows from \cite[Thm. 3.7]{dana-marius}, which is a more general statement about Fell bundle $C^*$-algebras. Things
are more interesting if we consider the reduced crossed product. Before going any further, we need to show that $\iota$ and $q$ induce
homomorphisms at the level of reduced crossed products.

\begin{lem}
\label{lem:EquivRepReduced}
	Let $(\A, G, \alpha)$ and $(\B, G, \beta)$ be separable groupoid dynamical systems, and let $\varphi : A \to B$ be a $G$-equivariant
	homomorphism. Given $u \in \go$, suppose $\rho$ is a nondegenerate separable representation of $B(u)$ on a Hilbert space $\Hil$,
	$\pi = \rho \circ q$ is the corresponding representation of $B$ on $\Hil$, and let $\Hil_\ess$ denote the essential subspace of the possibly 
	degenerate representation $\pi \circ \varphi$. Then $L^2(G_u, \Hil, \lambda_u) \otimes \Hil_\ess$ is the essential subspace of 
	$L_\pi \circ (\varphi \rtimes \id)$. Moreover,
	\[
		L_{(\pi \circ \varphi)_\ess} = \bigl( L_\pi \circ (\varphi \rtimes \id) \bigr)_\ess.
	\]
\end{lem}
\begin{proof}
	Clearly $L^2(G_u, \lambda_u) \otimes \Hil_\ess$ embeds isometrically into $L^2(G_u, \lambda_u) \otimes \Hil$. For any $h \in 
	L^2(G_u, \lambda_u) \otimes \Hil_\ess$, we have
	\begin{align*}
		L_\pi \circ (\varphi \rtimes \id)(f)h(\gamma) 
			&= \int_G \rho \bigl( \beta_\gamma^{-1} \bigl( \varphi_{r(\eta)}(f(\eta)) \bigr) \bigr) h(\eta^{-1}\gamma)  \, 
				d\lambda^{r(\gamma)}(\eta) \\
			&= \int_G \rho \bigl( \varphi_{s(\gamma)} \bigl(\alpha_\gamma^{-1} (f(\eta)) \bigr) \bigr) h(\eta^{-1}\gamma)  \, 
				d\lambda^{r(\gamma)}(\eta) \\
			&= \int_G \rho \circ \varphi_u \bigl( \alpha_\gamma^{-1}(f(\eta)) \bigr) h(\eta^{-1}\gamma) \, d\lambda^{r(\gamma)}(\eta) \\
			&= \int_G (\pi \circ \varphi)_u \bigl( \alpha_\gamma^{-1}(f(\eta)) \bigr) h(\eta^{-1}\gamma) \, d\lambda^{r(\gamma)}(\eta) \\
			&= L_{\pi \circ \varphi}(f)h(\gamma).
	\end{align*}
	Thus $L^2(G_u, \lambda_u) \otimes \Hil_\ess$ is invariant for $L_\pi \circ (\varphi \rtimes \id)$. Now suppose that $g \otimes k \in 
	(L^2(G_u, \lambda_u) \otimes \Hil_\ess)^\perp$. Then in particular,
	\[
		\ip{g \otimes k}{g \otimes h} = \ip{g}{g} \ip{k}{h} = 0
	\]
	for all $h \in \Hil_\ess$, so $k \in \Hil_\ess^\perp$. Therefore,
	\begin{align*}
		L_\pi \circ (\varphi \rtimes \id)(f)(g \otimes k)(\gamma) &= \int_G \rho \bigl( \beta_\gamma^{-1} \bigl( \varphi \rtimes \id(f)(\eta) \bigr) \bigr)
				g(\eta^{-1} \gamma) k \, d\lambda^{r(\gamma)}(\eta) \\
			&= \int_G g(\eta^{-1}\gamma) (\pi \circ \varphi)_u \bigl( \alpha_\gamma^{-1}(f(\eta)) \bigr) k \, d\lambda^{r(\gamma)}(\eta) \\
			&= 0
	\end{align*}
	since $k$ is in the zero space of $\pi \circ \varphi$. Thus $L^2(G_u, \lambda_u) \otimes \Hil_\ess$ is a nondegenerate invariant subspace for 
	$L_\pi \circ (\varphi \rtimes \id)$, and this representation vanishes on its orthogonal complement. Therefore, $L^2(G_u, \lambda_u) \otimes 
	\Hil_\ess = (L^2(G_u, \lambda_u) \otimes \Hil)_\ess$, and the computations above show that $L_\pi \circ (\varphi \rtimes \id) = L_{\pi \circ 
	\varphi}$ on this subspace.
\end{proof}

\begin{prop}
\label{prop:GequivReducedHoms}
	Let $(\A, G, \alpha)$ and $(\B, G, \beta)$ be separable groupoid dynamical systems, and let $\varphi : A \to B$ be a $G$-equivariant
	homomorphism. Then there is a homomorphism $\varphi \rtimes \id : \A \rtimes_{\alpha, r} G \to \B \rtimes_{\beta, r} G$ taking 
	$\gcra$ into $\Gamma_c(G, r^*\B)$ and satisfying
	\[
		\varphi \rtimes \id(f)(\gamma) = \varphi_{r(\gamma)} \bigl( f(\gamma) \bigr)
	\]
	for $f \in \gcra$.
\end{prop}
\begin{proof}
	We know from Proposition \ref{prop:Gequivariant} that such a map $\varphi \rtimes \id : \gcra \to \Gamma_c(G, r^*\B)$ exists, and that
	it is a $*$-homomorphism. Therefore, we just need to see that it extends to a map between the reduced crossed products. For each 
	$u \in \go$, let $\rho$ be a faithful representation of $B(u)$, and put $\pi_u = \rho_u \circ q_u$. Then each $\pi_u \circ \varphi$ is a 
	representation of $A$, so we clearly have
	\[
		\norm{\Ind(\pi_u \circ \varphi)(f)} \leq \norm{f}_r
	\]
	for all $f \in \gcra$. However, Lemma \ref{lem:EquivRepReduced} tells us that $\Ind(\pi_u \circ \varphi)(f) = (\Ind \pi_u) \circ (\varphi \rtimes 
	\id)(f)$, so
	\[
		\norm{\varphi \rtimes \id(f)}_r = \sup_{u \in \go} \norm{(\Ind \pi_u) \circ (\varphi \rtimes \id)(f)} = \sup_{u \in \go} 
			\norm{\Ind(\pi \circ \varphi)(f)} \leq \norm{f}_r
	\]
	for all $f \in \gcra$. Thus $\varphi \rtimes \id$ is bounded with respect to the reduced norms on $\gcra$ and $\Gamma_c(G, r^*\B)$,
	so it extends to a map $\varphi \rtimes \id : \A \rtimes_{\alpha,r} G \to \B \rtimes_{\beta, r} G$.
\end{proof}

Since $q \rtimes \id$ maps $\gcra$ onto $\Gamma_c(G, r^*\B)$ by \cite[Lem. 6.3.2]{ananth-renault}, $q \rtimes \id$ has dense range and is thus 
surjective. It is also true that $\iota \rtimes \id$ gives an embedding of $\I \rtimes_{\alpha, r} G$ into $\A \rtimes_{\alpha, r} G$ as an ideal. To see 
this, let $\{\rho_u\}$ be a family of faithful representations of the fibers of $A$, and put $\pi_u = \rho_u \circ q_u$ for all $u$. Then each 
$\pi_u \circ \iota$ is a (possibly degenerate) representation of $I$, and Lemma \ref{lem:EquivRepReduced} tells us that
\[
	\norm{\Ind(\pi_u \circ \iota)(f)} = \norm{(\Ind \pi_u) \circ (\iota \rtimes \id)(f)}
\]
for all $f \in \Gamma_c(G, r^*\I)$. However, $\pi_u \circ \iota$ is just the lift of $\rho_u \circ \iota_u$ to $I$, and the latter is a faithful representation
of $I(u)$. Therefore,
\[
	\norm{f}_r = \sup_{u\in \go} \norm{\Ind(\pi_u \circ \iota)(f)} = \sup_{u \in \go} \norm{(\Ind \pi_u) \circ (\iota \rtimes \id)(f)} = \norm{\iota \rtimes 
		\id(f)}_r
\]
for all $f \in \Gamma_c(G, r^*\I)$, so $\iota \rtimes \id$ is isometric. We also have $\image(\iota \rtimes \id) \subset \ker(q \rtimes \id)$, 
so we get a sequence
\begin{equation}
\label{eq:ReducedExactSequence}
		\xymatrix{
		0 \ar[r] & \I \rtimes_{\alpha \vert_I, r} G \ar[r]^{\iota \rtimes \id} & \A \rtimes_{\alpha, r} G \ar[r]^(0.43){q \rtimes \id} & 
			\A/\I \rtimes_{\alpha^I, r} G \ar[r] & 0.
		}	
\end{equation}
However, \eqref{eq:ReducedExactSequence} not exact in general. Gromov has famously produced examples of \emph{groups} for 
which \eqref{eq:ReducedExactSequence} fails to be exact. There are more tractable (but still complicated) examples of groupoids that cause 
\eqref{eq:ReducedExactSequence} to go bad. Given the unfortunate existence of such groupoids, it makes sense to single out the ones for 
which \eqref{eq:ReducedExactSequence} is always exact. 

\begin{defn}
	A second countable locally compact groupoid $G$ is said to be \emph{exact} if whenever $(\A, G, \alpha)$ is a separable 
	groupoid dynamical system and $I$ is a $G$-invariant ideal in $A$, the sequence
	\begin{equation}
	\label{eq:invariantsequence2}
		\xymatrix{
			0 \ar[r] & \I \rtimes_{\alpha \vert_I, r} G \ar[r]^{\iota \rtimes \id} & \A \rtimes_{\alpha, r} G \ar[r]^(0.45){q \rtimes \id} & 
				\A/\I \rtimes_{\alpha^I, r} G \ar[r] & 0.
		}
	\end{equation}
	is short exact.
\end{defn}

Note that we have the following immediate corollary to \cite[Lem. 6.3.2]{ananth-renault} when $G$ is measurewise amenable.

\begin{cor}
	If $G$ is a measurewise amenable secound countable locally compact Hausdorff groupoid, then $G$ is exact.
\end{cor}
\begin{proof}
	If $G$ is measurewise amenable and $(\A, G, \alpha)$ is a separable dynamical system, then the sequence \eqref{eq:invariantsequence}
	coincides with \eqref{eq:invariantsequence2}. But this sequence of full crossed products is exact, so it follows that $G$ is exact.
\end{proof}

We now have almost all of the pieces in place to prove the promised exactness theorem. However, note that all the work that we have done so 
far applies only to \emph{separable} $C^*$-algebras. Since we'll need to take spatial tensor products with short exact sequences of 
arbitrary $C^*$-algebras, we seem to be in some trouble. We had a similar problem for the nuclearity theorem, which we were able 
to circumvent via Proposition \ref{prop:sepnuclear}. Fortunately, a similar trick will work here. For brevity, we will say that a $C^*$-algebra 
$A$ is \emph{separably exact} if whenever
	\[
		\xymatrix{
			0 \ar[r] & I \ar[r]^(0.47){\iota} & B \ar[r]^(0.43){q} & B/I \ar[r] & 0
		}
	\]
is a short exact sequence of \emph{separable} $C^*$-algebras, the sequence
	\[
		\xymatrix{
			0 \ar[r] & I \ar[r]^(0.45){\iota \otimes \id} & B \ar[r]^(0.45){q \otimes \id} & B/I \ar[r] & 0
		}
	\]
is exact. Our aim is to show that $A$ is exact if and only if it is separably exact. For the proof we will need an alternative characterization of exactness, 
which was observed by Kirchberg in \cite[Thm. 1.1]{kirchberg83}: a $C^*$-algebra $A$ is exact if and only if 
\begin{equation}
\label{eq:calkin}
	\xymatrix{
		0 \ar[r] & \K(\Hil) \osig A \ar[r]^{\iota \otimes \id} & B(\Hil) \osig A \ar[r]^(0.42){q \otimes \id} & B(\Hil)/\K(\Hil) \osig A \ar[r] & 0 \\
	}
\end{equation}
is exact, where $\Hil$ is a separable infinite-dimensional Hilbert space.

\begin{prop}
\label{prop:sepexact}
	If a $C^*$-algebra $A$ is separably exact, then it is exact.
\end{prop}
\begin{proof}
	Suppose $A$ is an inexact $C^*$-algebra. Then the sequence \eqref{eq:calkin} is not exact. Therefore, there is an 
	$x \in \ker(q \otimes \id)$ that does not belong to $\K(\Hil) \osig A$. We can approximate $x$ by a sequence $\{t_i\}$, where each
	\[
		t_i = \sum T_{ij} \otimes a_{ij}
	\]
	is a sum of elementary tensors. Let $B$ be the separable $C^*$-subalgebra of $B(\Hil)$ generated by the $T_{ij}$ and $\K(\Hil)$. Since 
	$B \osig A$ sits naturally inside  $B(\Hil) \osig A$ as a $C^*$-subalgebra, we have $x \in B \osig A$. Then we obtain a sequence
	\[
		\xymatrix{
			0 \ar[r] & \K(\Hil) \osig A \ar[r]^(0.54){\iota \otimes \id} & B \osig A \ar[r]^(0.39){q \otimes \id} & B/\K(\Hil) \osig A \ar[r] & 0
		}
	\]
	which is not exact by construction, since $q \otimes \id$ is simply the restriction of the original quotient map to $B \osig A$. Therefore, 
	$x \in B \osig A$ implies that $\ker(q \otimes \id) \neq \K(\Hil) \osig A$, and $A$ cannot be separably exact.
\end{proof}

\begin{thm}
	Let $(\A, G, \alpha)$ be a separable groupoid dynamical system with $A$ exact and $G$ exact. Then the reduced crossed 
	product $\A \rtimes_{\alpha,r} G$ is exact.
\end{thm}
\begin{proof}
	By Proposition \ref{prop:sepexact} it is enough to consider separable $C^*$-algebras. Let
	\[
		\xymatrix{
			0 \ar[r] & I \ar[r] &  B \ar[r] &  B/I \ar[r] & 0
		}
	\]
	be a short exact sequence of separable $C^*$-algebras. By Theorem \ref{thm:reducedexchange}, we have an isomorphism 
	$(\A \rtimes_{\alpha, r} G) \osig B \cong (\A \osig B) \rtimes_{\alpha \otimes \id, r} G$. It is easy to see that $A \osig I$ is 
	$\alpha \otimes \id$-invariant, so Theorem \ref{thm:reducedexchange} also yields isomorphisms
	\[
		(\A \rtimes_{\alpha, r} G) \osig I \cong (\A \osig I) \rtimes_{\alpha \otimes \id, r} G
	\]
	and
	\[
		(\A \rtimes_{\alpha, r} G) \osig B/I \cong (\A \osig B/I) \rtimes_{\alpha \otimes \id, r} G.
	\]
	It is straightforward to check that the diagram
	\[
		\xymatrix{
			0 \ar[r] & (\A \rtimes_{r} G) \osig I \ar[r] \ar[d] & (\A \rtimes_{r} G) \osig B \ar[r] \ar[d] & (\A \rtimes_{r} G) \osig B/I \ar[r] \ar[d] 
			& 0 \\
			0 \ar[r] & (\A \osig I) \rtimes_{r} G \ar[r] & (\A \osig B) \rtimes_{r} G \ar[r] & (\A \osig B/I) \rtimes_{r} G \ar[r] & 0
		}
	\]
	commutes, so we have an isomorphism of short exact sequences. It suffices to consider the exactness of the second one. Since $A$ 
	is exact, the sequence
	\[
		\xymatrix{
			0 \ar[r] & A \osig I \ar[r] & A \osig B \ar[r] & A \osig B/I \ar[r] & 0
		}
	\]
	is exact. Since $G$ is assumed to be exact, this sequence remains exact after taking the reduced crossed product by 
	$G$. Consequently, $\A \rtimes_{\alpha, r} G$ is exact.
\end{proof}

\begin{cor}
	Let $(\A, G, \alpha)$ be a separable groupoid dynamical system with $\A$ exact and $G$ measurewise amenable. Then $\A \rtimes_\alpha G$ is exact.
\end{cor}
\begin{proof}
	Since $G$ is amenable, $\A \rtimes_\alpha G = \A \rtimes_{\alpha, r} G$, and the latter is exact by the previous theorem.
\end{proof}

\subsection*{Acknowledgements} This work was completed as part of the author's doctoral dissertation. I would like to thank Dana Williams for 
his encouragement, and for helpful comments on the preliminary versions of this paper.

\bibliographystyle{amsplain}
\bibliography{}
\end{document}